\documentclass[a4paper,11pt]{amsart}
\usepackage[all]{xy}
\usepackage{stmaryrd}
\usepackage{hyperref}
\usepackage{enumerate}
\usepackage{amssymb}

\newtheorem{thm}{Theorem}[section]
\newtheorem*{thmz}{Theorem}             
\newtheorem{prop}[thm]{Proposition}
\newtheorem{lem}[thm]{Lemma}

\theoremstyle{definition}
\newtheorem{defn}[thm]{Definition}
\newtheorem{ex}[thm]{Example}

\theoremstyle{remark}
\newtheorem{rem}[thm]{Remark}

\newcommand{\RR}{\mathbb R}                                         

\newcommand{\tto}{\rightrightarrows}                        

\newcommand{\F}{\ensuremath{\mathcal{F}}}

\newcommand{\N}{\ensuremath{\mathcal{N}}}

\DeclareMathOperator{\Lie}{\mathcal{L}}         
\renewcommand{\d}{\mathrm d}                    
\newcommand{\X}{\ensuremath{\mathfrak{X}}}          
\newcommand{\D}{\mathcal D}                      

\DeclareMathOperator{\Ver}{Vert}        

\newcommand{\G}{\mathcal{G}}            
\newcommand{\M}{\mathcal{M}}            
\newcommand{\C}{\mathcal{C}}            
\renewcommand{\O}{\mathcal{O}}
\newcommand{\s}{\mathbf{s}}             
\renewcommand{\t}{\mathbf{t}}           
\renewcommand{\F}{\mathcal{F}}          

\newcommand{\K}{\mathcal K}

\newcommand{\g}{\mathfrak{g}}           

\DeclareMathOperator{\im}{Im}           
\DeclareMathOperator{\Curv}{Curv}

\DeclareMathOperator{\ad}{ad}           
\DeclareMathOperator{\Aut}{Aut}         
\DeclareMathOperator{\Der}{Der}         
\DeclareMathOperator{\Z}{Z}
\DeclareMathOperator{\Gau}{Gauge}

\newcommand{\ri}{\overrightarrow}

\usepackage[dvips]{graphicx}

\newcommand{\comment}[1]{}
\begin{document}
 \title{Extensions of Lie Brackets.}
\author{Olivier Brahic}

\address{Depart.~de Matem\'{a}tica,
Instituto Superior T\'{e}cnico, 1049-001 Lisboa, PORTUGAL}
\email{brahic@math.ist.utl.pt}
\thanks{Partially supported by the Funda\c{c}\~ao para a Ci\^encia e a Tecnologia
through the Program POCI 2010/FEDER and by the Projects POCI/MAT/57888/2004 and POCI/MAT/55958/2004.}

\begin{abstract}
We provide a framework for extensions of Lie algebroids, including non-abelian extensions and
Lie algebroids over different bases. Our approach involves Ehresmann
connections, which allows straight generalizations of classical
constructions. We exhibit a filtration in cohomology and explain the associated spectral sequence.
 We also give a description of the groupoid integrating an extension in case a complete connection exists.
 The integrability is also studied.
\end{abstract}

\maketitle

\section*{Introduction}            %
\label{sec:introduction}           %
The concept of Lie algebroid proved to be relevant in diverse fields of geometry, with interesting applications in physics as well.
 However, despite intensive studies, the notion of extension was never systematically investigated.
 It is the aim of this paper to fill this gap. We study extensions of Lie algebroids from a general perspective,
 the interest is to get a unified treatment for a large variety of situations.

The basic setting will be a surjective Lie algebroid morphism $\pi:A_E\to A_B$ covering a surjective submersion
$p:E\to B$ where $A_E$ and $A_B$ are Lie algebroids over $E$ and $B$, respectively. In this
situation, we denote $\K$ the kernel of $\pi$. It is a Lie algebroid over $E$ and we obtain this
way $A_E$ as an extension of $A_B$ by $\K$:
$$ \K\hookrightarrow A_E \twoheadrightarrow A_B.$$

In Section \ref{generalities}, we recall basic notion concerning Lie algebroids and their extensions. Then we introduce in Section \ref{Ehresmann} the notion of Ehresmann connection for an extension: it
is given by a sub-vector bundle $H$ complementary to $\K$ in $A_E$:
$$ A_E=\K\oplus H.$$
 Such a connection always exists and we can construct the usual horizontal lifting, it is a $C^\infty(B)$-linear application $h:\Gamma(A_B)\to \Gamma(H)$. From these notions, we derive a couple $(\D,\omega)$ where:
\begin{itemize}
\item $\D:\Gamma(A_B)\to \Der(\K)$ is a $C^\infty(B)$-linear map with values in derivations of $\K$, such that the symbol $s_{\D_\alpha}$ of $\D_{\alpha}$ $p_*$-projects onto $\sharp_B(\alpha)$ for any $\alpha\in\Gamma(A_B)$;
\item $\omega$ is a $2$-form on $A_B$ with values in sections of $\K$: $\omega\in\Omega^2(A_B)\otimes\Gamma(\K)$.
\end{itemize}
In order to define an extension, the couple $(\D,\omega)$ is submitted to the following
compatibility conditions:
\begin{align*}
\partial_H\omega&=0,\\
 \Curv_\D&=\ad^\K\circ\,\omega.
\end{align*}
  Here, $\partial_H:\Omega^k(A_B)\otimes\Gamma(\Lambda^l \K)\to\Omega^{k+1}(A_B)\otimes\Gamma(\Lambda^l
\K)$ is a covariant derivative associated to $\D$ by usual formulas and
$\Curv_\D(\alpha,\beta):=[\D_\alpha,\D_\beta]-\D_{[\alpha,\beta]}$.

The description of extensions we get this way lies in-betwen usual treatments for both Lie algebras and fibrations:
\begin{thmz}
Let $p:E\to B$ be a submersion, $\K\to E$ a Lie algebroid  whose characteristic foliation lies in $\Ver$ and  $A_B\to B$ a Lie algebroid.

Extensions of $A_B$ by $\K$ covering $p$ are (up to equivalence) classified by equivalence classes of admissible couples $(\D,\omega)$ under the relation $(\D,\omega)\sim(\D',\omega')$ if there exists $\Delta\in\Omega^1(A_B)\otimes \Gamma(\K)$ such that
$$\begin{array}{ccl}
\D'&=&\D+\ad^\K\circ{\Delta},\\
 \omega'&=&\omega+\partial_H\Delta +[\Delta\wedge\Delta]_\K,
\end{array}$$
where $[\Delta\wedge\Delta]_\K\in\Omega^2(A_B)\otimes\Gamma(\K)$ is defined by: $$[\Delta\wedge\Delta]_\K(\alpha,\beta)=[\Delta(\alpha),\Delta(\beta)]_\K.$$
\end{thmz}

Besides, we explain that given a connection, parallel transport along an $A_B$-path $a\in P(A_B)$ is well
defined (see subsection \ref{parallel}) and, in the case of a complete connection, we obtain the usual notion of holonomy along $a$ as a Lie algebroid morphism
$$\xymatrix{
\K|_{E_{\gamma(0)}} \ar[r]^{\Phi_a}  & \K|_{E_{\gamma(1)}}.
}$$
Here, $\gamma:I\to B$ is the base path of $a$. In general, $\Phi$ does not factor through $A_B$-homotopies unless $\Curv_\D$ vanishes. Note however the relation $\Curv_\D=\ad^\K\circ\,\omega$ suggests that this can by measured by inner automorphisms of $\K$, this will play an important role. In any case, one can still think of $\Phi$ as a groupoid morphism $\Phi:P(A_B)\to
\Gau(\K)$ where 
$${\Gau(\K):=\bigl\{\text{Lie algebroid morphisms }\Psi_{y,x}:
\K|_{E_x}\to \K|_{E_y}\bigr\}}$$ and $P(A_B)$ are both groupoids over $B$.

In Section \ref{Cohomology} we look at the cohomology of an extension. There is a natural filtration, and the choice of a connection allows a rather nice description of the associated spectral sequence.

The Section \ref{integration} is devoted to the description of the Weinstein groupoid $\G(A_E)$. We obtain $\G(A_E)$ as a quotient of $P(A_B)\ltimes \G(\K)$. Here $P(A_B)$ acts on $\G(\K)$ by means of holonomy, and $P(A_B)\ltimes \G(\K)$ corresponds to an associated groupoid over $E$. The precise statement is given by Theorem \ref{thm-2cocycle}:
\begin{thmz} Consider a Lie algebroid extension $\K\hookrightarrow A_E\twoheadrightarrow A_B$
endowed with a Ehresmann connection $A_E=\K\oplus H$, which we assume to be complete.

Then the topological source simply connected groupoid $\G({A_E})$ integrating ${A_E}$ is naturally
identified with the quotient
$$P(A_B)\ltimes_B \G({\K})/{\bf{\sim}},$$
where the equivalence relation is given by: $(a_0,g_0){\bf{\sim}}(a_1,g_1)$ if and only if there exists a
${A_B}$-homotopy $h_B=a dt+b d\epsilon:TI^2\mapsto {A_B}$ between $a_0$ and $a_1$ such that:
$$g_1\cdot g_0^{-1}=\partial (h_B,\t(g_0)).$$ Here, $\partial(h_B, x_0)$ is the element in $\G(\K)$
represented by the ${\K}$-path:
\begin{equation*}\epsilon\rightarrow\int_0^1 (\Phi^{a_B^\epsilon}_{0,s})_*\bigl( \omega(a,b)_{s,\epsilon}\bigr)ds\in {\K}_{{\gamma_K}^\epsilon},
\end{equation*}
where $\gamma_K^\epsilon:=\phi^{-1}_{a_B^\epsilon}\circ \phi_{a_B^0}(x_0).$
\end{thmz}

In general, the sequence of topological groupoids 
\begin{equation*} 1\to \G(\K)\xrightarrow{\tilde{\iota}} \G(A_E) \xrightarrow{\tilde{\pi}} \G(A_B)\to 1.
\end{equation*}
obtained by integration of a Lie algebroid extension might not be exact. In order to measure this lack of exactness, we denote $\M$ the monodromy groupoid of the extension. It is defined as the kernel of $\tilde{\iota}$. In the case there exists a complete Ehresmann connection, we can build a connecting homomorphism whose image is precisely $\M$, similarly to what is well known for fibrations:
\begin{thmz} Let  $\K\hookrightarrow A_E \twoheadrightarrow A_B$ be a Lie algebroid extension that admits a complete Ehresmann connection. Then there exists a homomorphism
$$\partial_{2}:\pi_2(A_B)\ltimes_{B} E\to \G(\K),$$
that makes the following sequence exact:
$$ \cdots \to \pi_2(A_B)\ltimes E\xrightarrow{\partial_2} \G(\K)\xrightarrow{\tilde{\iota}} \G(A_E)\twoheadrightarrow\G(A_B).$$
\end{thmz}
Here, $\pi_2(A_B)$ denote the second homotopy groups of $A_B$, as defined in the Appendix \ref{app:spheres}. In the case $A_B$ is integrable, it coincides with the union of all second homotopy groups of the source fibers, seen as
a bundle of groups over $B$. In this work we do not try to put any topology on $\pi_2(A_B)$ whose role is
essentially algebraic.

 The interest of the monodromy groupoid $\M$ lies in the fact that it controls the integrability of $A_E$ in the case of \emph{clean} extensions. By \emph{clean}, we mean by definition that the
restrictions of $\pi$ to the isotropy Lie algebras of $A_E$ surject onto those of
$A_B$. This implies that, pointwise, one has exact sequences of Lie algebras
$$(\ker \sharp_\K)_{x}\hookrightarrow (\ker \sharp_E)_{x} \twoheadrightarrow (\ker \sharp_B)_{p(x)}.$$
In this case, we show that $A_E$ is integrable provided $A_B$ and $\K$ are, and if $\M$ is discrete
near identities (see \ref{discreteness} for the precise definition).

\begin{thmz} Let $\K\hookrightarrow A_E \twoheadrightarrow A_B$ be a clean Lie algebroid extension.
Assume that it admits a complete Ehresmann connection and that both $\K$ and $A_B$ are
integrable Lie algebroids.

Then $A_E$ is integrable if $\mathcal{M}$ is discrete near identities in $\G(\K)$.
\end{thmz}
\noindent \textbf{Acknowledgments.} I would like to thank Rui Loja Fernandes for the many useful
discussions and suggestions that helped to improve this work.

\newpage
\tableofcontents

\section{Generalities on Lie algebroids}\label{generalities}
We first briefly recall basic definitions concerning Lie algebroids \cite{Pr} and refer to \cite{CW}, \cite{M} for more details.

\begin{defn}\label{defalgebroid}
 A \textbf{Lie algebroid} $A\to M$ is a vector bundle $A$ over a smooth manifold $M$, together with vector bundle homomorphism $\sharp:A\to TM$ called the anchor, and a bracket $[\ ,\ ]_A$ defined on the space $\Gamma(A)$ of sections of $A$ such that for any smooth function $f\in C^\infty(M)$ and any section $\alpha, \beta \in\Gamma(A)$ of $A$, the following conditions hold
\begin{enumerate}[i)]
 \item $[\alpha,f\beta]_A=f[\alpha,\beta]_A+\Lie_{\sharp(\alpha)}(f) \beta$; \label{Leibniz}
 \item $[\alpha,\beta]_A=-[\beta,\alpha]_A$;
 \item $\sharp([\alpha,\beta]_A)=[\sharp(\alpha),\sharp(\beta)]_{TM}$; \label{anchor}
 \item $\displaystyle \oint_{\alpha,\beta,\gamma} \bigl[[\alpha,\beta]_A,\gamma\bigr]_A=0.$\label{Jacobi}
\end{enumerate}
In other words, the bracket $[\ ,\ ]_A$ induces a structure of real Lie algebra on sections of $A$ such that the Leibniz condition (\ref{Leibniz}) holds and $\sharp$ is a Lie algebra homomorphism from $\Gamma(A)$ to $\X(M):=\Gamma(TM)$.
\end{defn}
The cohomology associated to a Lie algebroid is defined on the complex of cochains
$C^k:=\Gamma(\Lambda^k A^*)$ which we might also denote $\Omega^k(A)$ (or $\Omega^k(M)$ for short
when $A=TM$). The differential operator is given by the usual formula
\begin{multline} \langle \d_A c,\alpha_0\dots\alpha_k\rangle:=\sum_{i=0\dots n}(-1)^i\Lie_{\sharp \alpha_i} \langle c,\alpha_0\dots \hat{\alpha_i}\dots\alpha_k\rangle
\\
 +\sum_{0\leq i<j\leq k} (-1)^{i+j}\langle c,[\alpha_i,\alpha_j]_A ,\alpha_0\dots\hat{\alpha_i}\dots \hat{\alpha_j}\dots\alpha_k\rangle.
\end{multline}

\begin{defn} Given two Lie algebroids $A_M\to M$ and $A_N\to N$, a \textbf{Lie algebroid homomorphism} $\Phi:A_M\to A_N$ over $\phi:M \to N$ is a vector bundle homomorphism $\Phi$ covering $\phi$ such that the map induced on forms: $$\Phi^*:\Omega^k (A_N)\to \Omega^k( A_M),$$
 defined by $\langle\Phi^*(u),a\rangle=\langle {u}_{\phi(m)},\Phi(a) \rangle$, commutes with the differential:
$$\d_{A_M}\circ \Phi^*=\Phi^*\circ \d_{A_N}.$$
\end{defn}
When $\phi$ is a diffeomorphism, we shall denote $\Phi_*:\Gamma(A)\to\Gamma(A)$ the application induced on sections, $\Phi_*(\alpha):=\Phi\circ\alpha\circ\phi^{-1}$. See for instance \cite{HM} for a detailed discussion on Lie algebroid morphisms.

Recall that the isotropy at a point $x\in M$ is defined as the kernel of the anchor, $\g_x^A:=\ker \sharp_x$. It is naturally endowed with a structure of Lie algebra. In general, the dimension of the isotropy algebra may vary from one point to another, so all isotropies do not fit in a smooth bundle of Lie algebras in general.
\subsection{Derivations.}\label{derivations}
\begin{defn}A \textbf{derivation} of a Lie algebroid $A\to M$ is an application
$$D:\Gamma(A)\to\Gamma(A),$$
together with a vector field $s_D\in\X(M)$ (called the \textbf{symbol} of the derivation) satisfying for
any smooth function $f\in C^\infty(M)$ and any two sections of $A$ $\alpha,\beta\in \Gamma(A)$
\begin{align*}
   D(f\alpha)            &= fD(\alpha)+s_D(f)\alpha,\\
   D([\alpha,\beta]_A)   &=  [D(\alpha),\beta]_A+[\alpha,D(\beta)]_A,\\
   \sharp(D(\alpha))   &= [s_D,\sharp(\alpha)]_{TM}.
  \end{align*}

\end{defn}
A derivation is the infinitesimal version of a Lie algebroid automorphism: any derivation
determines a vector field on  $A$ whose flow is a Lie algebroid automorphism \emph{et vice versa}
(see for instance the appendix in \cite{CrFe2}). We will denote $\Der(A)$ the space of all
derivations, it has a natural structure of Lie algebra, whose bracket is given by the commutator
$[D_1,D_2]=D_1\circ D_2-D_2\circ D_1$.
\begin{defn}
An \textbf{inner derivation} is a derivation of the form $D(\alpha)=[\delta,\alpha]_A$ for some smooth
section $\delta\in \Gamma(A)$, in which case we will write $D=\ad^A_{\delta}$.
\end{defn}
There are obvious notions of time-dependent derivations and inner derivations, corresponding to
which we get time-dependent vector fields on $A$. In particular, given a time-dependent section
$\alpha^t\in\Gamma(A)$, we shall denote $\psi^\alpha_{t,s}:A\to A$ the flow on $A$ associated to
 $\ad_{\alpha^t}^A$.

\subsection{Extensions of Lie algebroids}\label{extensions}%
The notion of extension we will use in this work is summarized below.

\begin{defn} An \textbf{extension} of a Lie algebroid $A_B$ is a surjective Lie algebroid morphism $\pi:A_E\to A_B$ covering a surjective submersion $p:E\to B$.
\end{defn}

With this definition,  $\K:=\ker\ \pi$ is a smooth sub-vector bundle of $A_E$ that, as proven below, turns out to be a Lie sub-algebroid of $A_E$. Thus, we obtain an exact sequence of Lie
algebroids:
$$ \K\hookrightarrow A_E \twoheadrightarrow A_B.$$
We shall refer to $A_E$ as an extension of $A_B$ by $\K$. As generic notations, we will write $\sharp_E:A_E\to TE $, $\sharp_B:A_B\to TB$, $\sharp_\K:\K\to
TE$ and $[\ ,\ ]_{A_E}$, $[\ ,\ ]_{A_B}$, $[\ ,\ ]_{\K}$ the respective anchors and brackets. The isotropy algebras will be denoted $\g_x^E:=\ker (\sharp_E)_x$, $\g_b^B:=\ker (\sharp_B)_b$,  and $\g_x^\K:=\ker (\sharp_\K)_x=\g_x^E\cap \K$, for any $x\in E,b \in B$.

\begin{rem}
 Note that the injection $\K\subset A_E$ does not determine $\pi$ (not even $p$) so one should not think of $A_B$ as a quotient of $A_E$ by $\K$.
\end{rem}

\begin{defn} Given an extension $\pi:A_E\twoheadrightarrow A_B$, we will refer to a \textbf{$\pi$-projectable} section $\mu$ of $A_E$ any section $\mu$ such that, for any form $\nu\in\Omega^1(A_B)$, the contraction $\langle\pi\circ\mu,\nu\rangle$ is a basic function on $E\to B$. This means that $\pi\circ\mu$ induces a well defined section of $A_B$ that we will denote $\pi(\mu)$.
\end{defn}

\begin{lem}\label{projectableanchor}
 Let $\mu\in \Gamma(A_E)$ be a $\pi$-projectable section, then $\sharp_E\circ\mu$ is $p_*$-projectable, moreover:
$$p_*(\sharp_E(\mu))=\sharp_B(\pi(\mu)). $$
\end{lem}
\begin{proof}
Recall the condition for $\pi$ to be a Lie algebroid morphism: $d_{A_E}\circ \pi^*=\pi^*\circ
d_{A_B}$ where $\pi^*:\Omega^k(A_B^*)\to \Omega^k(A_E)$ is naturally induced by $\pi$. For $k=0$,
this condition writes:
\begin{align*}
 &    &{\langle}\d_{A_E}\circ \pi^*(f),\mu{\rangle}&={\langle}\pi^*\circ \d_{A_B}(f),\mu{\rangle} & \quad (f\in C^\infty(B), \mu\in \Gamma(A_E))\\
&\iff & {\langle}\d (\pi^*f),\sharp_E(\mu){\rangle}&={\langle} \d_A(f), \pi\circ \mu{\rangle} \\
&\iff & {\langle}\d (f\circ p),\sharp_E(\mu){\rangle}&={\langle} \d f,\sharp_B\circ \pi\circ \mu{\rangle} \\
&\iff & {\langle}\d f\circ \d p,\sharp_E(\mu){\rangle}&= {\langle}\d f,\sharp_B \circ\pi\circ \mu{\rangle},\\
  \end{align*}
which is \emph{\`a priori} an equality in $C^\infty(E)$. However, if we assume $\mu$ to be $\pi$-projectable, the right-hand term is in fact a basic function, which implies that $\sharp_E(\mu)$ $p_*$-projects onto $\sharp_B(\pi(\mu))$.
\end{proof}
\begin{lem}\label{projectablebrackets}
 Let $\mu,\nu \in\Gamma(A_E)$ be $\pi$-projectable sections, then $[\mu,\nu]_{A_E}$ is also projectable, moreover:
$$ \pi([\mu,\nu]_{A_E})=[\pi(\mu),\pi(\nu)]_{A_B}.$$
\end{lem}
\begin{proof}
This is an immediate consequence of the condition $\d_{A_E}\circ \pi^*=\pi^*\circ \d_{A_B}$ $\pi^*:\Omega^k(A_B)\to \Omega^k(A_E)$ for $k=1$ and of the preceding lemma.
\end{proof}

As a consequence of the two last lemmas, we can state the following.

\begin{prop}\label{kernel}
Let $\pi:A_E\to A_B$ be a Lie algebroid extension. Then  $\K:=\ker\ \pi$ is a Lie algebroid
on $E$ whose characteristic foliation lies in $\Ver:=\ker \ p_*\subset TE$.
\end{prop}

In particular, if we restrict the vector bundle $\K$ to a fiber $E_x:=p^{-1}(\{x\})$ over any $x\in B$, we obtain a nice Lie algebroid $\K|_{E_x}$ over $E_x$.

\section{Ehresmann Connections}\label{Ehresmann}
We now introduce a notion of connection \cite{Ehr} for extensions of Lie algebroids.

\begin{defn}
 An \textbf{Ehresmann connection} on a Lie algebroid extension $\K\hookrightarrow A_E \twoheadrightarrow A_B$ is a smooth sub-vector bundle $H\subset A_E$ complementary to $\K$ in $A_E$:
$$\K\oplus H=A_E.$$
We will call $H$ the horizontal distribution of the connection.
\end{defn}
Clearly, such a connection always exists and there is an associated notion of \textbf{horizontal lifting} $h:\Gamma(A_B)\mapsto \Gamma(H)$: we define $h(\alpha)$ as the unique section of $H$ that $\pi$-projects on $\alpha$.

We might use the terminology of connection in a rather unristricted way, to refer either to the operator
$h$ or to the horizontal sub-bundle $H$.

\begin{defn}
 Given an Ehresmann connection, we define the associated \textbf{curvature $2$-form} $\omega\in \Omega^2(A_B)\otimes \Gamma(\K)$ (otherwise stated, all tensor products are to be taken over the commutative ring ${C^\infty(B)}$) by the following formula:
$$ \omega(\alpha,\beta):=h([\alpha,\beta]_{A_B})-[h(\alpha),h(\beta)]_{A_E}.$$
\end{defn}
The fact that $\omega$ takes values in $\Gamma(\K)$ is a direct consequence of lemma \ref{projectablebrackets}.
A connection also determines a $C^\infty(B)$-linear application $\D: \Gamma(A_B)\to \Der(\K)$ defined by:
$$\D_\alpha (\kappa):=[h(\alpha),\kappa]_{A_E}.$$ More precisely for any sections $\alpha, \beta\in \Gamma(A_B)$, any function $g\in\C^\infty(B)$ and any sections $\kappa, \kappa_1,\kappa_2\in\Gamma(\K)$:
\begin{eqnarray}
\label{derivationlinearity1} \D_{\alpha}+\D_{\beta}&=&\D_{\alpha+\beta},\\
\label{derivationlinearity2} \D_{g.\alpha}&=&g.\D_\alpha,\\
\label{derivationkernel}     \D_\alpha([\kappa_1,\kappa_2]_\K)&=&[\D_\alpha (\kappa_1),\kappa_2]_\K+[\kappa_1,\D_\alpha (\kappa_2)]_\K,\\
\label{derivationsymbol}      \sharp_K(\D_{\alpha}(\kappa))&=&[s_{\D_\alpha},\kappa]_{TE}.
\end{eqnarray}

Here, the first two equalities express the $C^\infty(B)$-linearity of $\D$; equality (\ref{derivationlinearity1}) is clear and (\ref{derivationlinearity2}) follows from Proposition \ref{kernel}. The last two equations express the fact that $\D$ has values in derivations of $\K$: equation (\ref{derivationkernel}) is a consequence of Jacobi identity (\ref{Jacobi}) and (\ref{derivationsymbol}) follows from (\ref{anchor}) in Definition \ref{defalgebroid}.

As a derivation of $\K$, each $\D_\alpha$ naturally extends to a derivation of the graded algebra
$\Gamma(\Lambda^k \K)$ of multi-sections of $\K$:
$$\D_\alpha (\kappa_1\wedge\dots\wedge\kappa_k):=\sum_{i=1\dots k} \kappa_1\wedge \dots\wedge {\D_\alpha(\kappa_i)}\wedge \dots \wedge\kappa_k,\shoveleft{(\kappa_i\in\Gamma(\K)),}$$
and by duality to multi-sections of $\Omega^k(\K)$:
$$\langle \D_\alpha \theta,\kappa\rangle:=\Lie_{s_{\D_\alpha}}\langle\theta,\kappa\rangle-\langle  \theta,\D_\alpha\kappa\rangle, \shoveleft{\bigl(\theta\in\Omega^k(\K),\ \kappa\in\Gamma(\Lambda^k \K)\bigr).}$$

The covariant differential associated to $\D$ will be denoted
$\partial_H:\Omega^k(A_B)\otimes\Gamma(\Lambda^l \K)\to\Omega^{k+1}(A_B)\otimes\Gamma(\Lambda^l
\K)$ to fit with classical notations (note however that it is related to $\D$ rather than $H$
since two different horizontal distributions can induce the same derivations). It is given by the
usual formula:
\begin{multline}
\partial_H \theta\ (\alpha_0\dots\alpha_k):=\sum_{i=0\dots k} (-1)^{i}\D_{\alpha_k}\  \theta(\alpha_0\dots\hat{\alpha_i}\dots\alpha_k)\\
+\sum_{0\leq i<j\leq k}(-1)^{i+j}\theta([\alpha_i,\alpha_j]_{A_B}\alpha_0\dots\hat{\alpha}_i\dots\hat{\alpha}_j\dots\alpha_k).
\end{multline}
This exact formula can be dualized to define a covariant derivative
$\partial_H:\Omega^k(A_B)\otimes\Omega^l(\K)\to\Omega^{k+1}(A_B)\otimes\Omega^l(\K)$
which we will also denote $\partial_H$.

The curvature $\Curv_{\D}$ of $\D$ is defined by
$\Curv_{\D}(\alpha,\beta):=\D_\alpha\circ\D_\beta-\D_\beta\circ\D_\alpha-\D_{[\alpha,\beta]_{A_B}}$.
By construction, it is related to the curvature $2$-form $\omega$ by:
\begin{equation}\label{curvatures} \Curv_{\D}(\alpha,\beta)(\kappa)=[\omega(\alpha,\beta),\kappa]_{\K}.
\end{equation}
Moreover, $\omega$ is $\D$ closed in the sense that $\partial_H\omega=0$, which writes:
\begin{equation}\label{closedcurvatureform} \oint_{\alpha,\beta,\gamma} \D_{\alpha}\ \omega(\beta,\gamma)-\omega([\alpha,\beta]_{A_B},\gamma), \quad
(\alpha,\beta,\gamma\in\Gamma(A_B)).
\end{equation}
It is a consequence of Jacobi identity (\ref{Jacobi}).

In order not to confuse $\Curv_{\D}$ with $\omega$ we shall refer to $\omega$ as the \emph{curvature
$2$-form}, and to $\Curv_{\D}$ as the \emph{curvature}.

\begin{rem} The following fact will play an important role in this work: in general, $\D$ might fail to induce an infinitesimal action of $A_B$ on $E$. However, it is clear in \eqref{curvatures} that this failure is still controlled by $\K$. In other words, if we think of $\K$ as a defining a homotopy theory on $E$, one can see an extension of $A_B$ by $\K$ as an infinitesimal action up to homotopy. Note that this terminology corresponds to a different notion than the one of ``representations up to homotopy'' appearing for instance in \cite{CrFe3}.
\end{rem}

\subsection{Parallel transport}\label{parallel}

Given a connection, parallel transport along a small $A_B$-path is well defined, at least locally
as we briefly explain now.

Fix a time-dependent section $\alpha^t\in \Gamma(A_B)$. We denote $X_{\D_\alpha}^t$ the
 linear vector field on $\K$ associated to $\D_{\alpha^t}$ and $\phi^{\D_\alpha}_{t,s}$ its flow. It is \emph{\`a priori} only defined in a neighborhood of $\kappa$ in $\K$ and for $t$ close to $s$.

 Now, since the application $\alpha\mapsto \D_\alpha$ is
$C^\infty(B)$-linear, we see that the restriction  of $X_{D_\alpha}^t$ to a fiber $\K_{|E_b}$ over some $b\in B$ only
depends on $\alpha^t(b)$. Moreover, by definition $\D_\alpha^t=\ad^{A_E}_{h(\alpha^t)}$ so
it has symbol $\sharp_E(h(\alpha^t))$ which we know from Lemma \ref{projectableanchor}
to $p_*$-project onto $\sharp_B(\alpha^t)$. This implies that the restriction of
$\phi_{t,0}^{\D_\alpha}$ to $\K_{|E_b}$ only depends on $\alpha^{t'}$ evaluated at the point ${\phi_{t',0}^{\sharp_B\alpha}(b)}$ for  $t'$ in $[0,t]$.

Thus, if $a:I\to A_B$ is a $A_B$-path over $\gamma:I\to B$, one can extend it arbitrarily to a
time dependent $\alpha^t$ section of $A_B$ and define
$\Phi^a_{t,0}(\kappa):=\phi_{t,0}^{\D_\alpha}(\kappa)$ where $\kappa\in
\K_{|E_{\gamma(0)}}$. We know that it is well defined for $t$ small enough (depending on
$\kappa$) and independent of the choice of $\alpha^t$.

\begin{defn}
An Ehresmann connection is said to be \textbf{complete} if, for any $A_B$-path $a$, $\Phi_{t,0}^{a}(\kappa)$ is
well defined for any $\kappa\in\K_{|E_{\gamma(0)}}$ and all $t\in I$.
\end{defn}
For a complete connection, we obtain at time $t=1$ an application:
$$\xymatrix{
\K_{|E_{\gamma(0)}} \ar[r]^{\Phi_a}  & \K_{|E_{\gamma(1)}}.
}$$
called parallel transport or \textbf{holonomy} along $a$. Clearly $\Phi_a$ is a Lie algebroid
morphism. We shall denote $\phi_a:E_{\gamma(0)}\to E_{\gamma(1)}$ the base map covered by $\Phi^a$ 


\subsection{Admissible couples}
Once fixed a connection, one can decompose any section $\eta$ of $A_E$ as a sum of
elements of the form $\kappa+f h(\alpha)$ with $\kappa\in \Gamma(\K), \alpha\in\Gamma(A_B)$ and
$f\in C^\infty(E)$ so we get an isomorphism: 
$$\Gamma(A_E)=\Gamma(\K)\oplus\ C^\infty(E)\!\otimes\!
\Gamma(A_B).$$ The structure of Lie algebroid on $A_E$ is thus entirely determined by the ones on
$\K$ and $A_B$, and by the couple $(\D,\omega)$. Indeed, the brackets are given by:
\begin{eqnarray}
\label{splitbrackets1}\text{[}\kappa_1,\kappa_2\text{]}_{A_E}&=&\text{[}\kappa_1,\kappa_2\text{]}_\K,\\
\label{splitbrackets2}\text{[}h(\alpha),\kappa\text{]}_{A_E}&=&\D_\alpha\kappa,\\
\label{splitbrackets3}\text{[}h(\alpha),h(\beta)\text{]}_{A_E}&=&h(\text{[}\alpha,\beta\text{]}_{A_B})+\omega(\alpha,\beta),
\end{eqnarray}
and extended to arbitrary sections using Leibniz rule. The anchor is given by the following formulas:
$$\begin{array}{ccc}
\sharp_{A_E}(h(\alpha))&=&s_{\D_\alpha},\\
\sharp_{A_E}(\kappa)&=&\sharp_\K(\kappa).
\end{array}$$

Reciprocally, once fixed a submersion $p:E\to B$, a Lie algebroid $\K$ over $E$ whose characteristic
foliation lies in $\Ver$ and a Lie algebroid $A_B$ over $B$, one shall consider couples $(\D,\omega)$
where:
\begin{itemize}
\item $\D:\Gamma(A_B)\to \Der(\K)$ is a $C^\infty(B)$-linear application such that the symbol $s_{\D_\alpha}$ of $\D_{\alpha}$ $p_*$-projects onto $\sharp_B(\alpha)$ for any $\alpha\in\Gamma(A_B)$;
\item $\omega$ is a $2$-form on $A_B$ with values in sections of $\K$: $\omega\in\Omega^2(A_B)\otimes\Gamma(\K)$.
\end{itemize}
\begin{defn} The couple $(\D,\omega)$ is said to be admissible if  $\partial_H\omega=0$ and $\Curv_\D=\ad^\K\circ\,\omega$.
\end{defn}

Under the assumption of admissibility, the formulas above define a Lie algebroid structure on
$A_E:=\K\oplus p^*A_B$ coming with the obvious Ehresmann connection. Note however that given
arbitrary Lie algebroids $A_B$ and $\K$, there might not exist an extension of $A_B$ by $\K$.

A natural question is to ask how the choice of the connection influences our construction. Given
two connections $h$ and $h'$, one can form the difference $\Delta:=h-h'\in \Omega(A_B)\otimes
\Gamma(\K)$ and one easily gets the following relations:
\begin{eqnarray}\label{extensionequivalence}
\D'_\alpha\kappa&=&\D_\alpha\kappa+[\Delta(\alpha),\kappa]_{\K},\\
 \omega'(\alpha,\beta)&=&\omega(\alpha,\beta)+\partial_H\Delta(\alpha,\beta)+[\Delta(\alpha),\Delta(\beta)]_{\K}.
\end{eqnarray}
where of course $(\D,\omega)$ and $ (\D',\omega')$ denote the compatible couples respectively
associated to $h$ and $h'$. Thus one can state the following.

\begin{thm}\label{classification}
Let $p:E\to B$ be a submersion, $\K\to E$ a Lie algebroid  whose characteristic foliation lies in $\Ver$ and  $A_B\to B$ a Lie algebroid.

Extensions of $A_B$ by $\K$ covering $p$ are (up to equivalence) classified by equivalence classes of admissible couples $(\D,\omega)$ under the relation $(\D,\omega)\sim(\D',\omega')$ if there exists $\Delta\in\Omega^1(A_B)\otimes \Gamma(\K)$ such that
$$\begin{array}{ccl}
\D'&=&\D+\ad^\K\circ{\Delta},\\
 \omega'&=&\omega+\partial_H\Delta +[\Delta\wedge\Delta]_\K,
\end{array}$$
where $[\Delta\wedge\Delta]_\K\in\Omega^2(A_B)\otimes\Gamma(\K)$ is defined by: $$[\Delta\wedge\Delta]_\K(\alpha,\beta)=[\Delta(\alpha),\Delta(\beta)]_\K.$$
\end{thm}
\begin{rem}
 As for usual considerations concerning extensions, by equivalence of extensions, we mean an algebroid isomorphism $A_E\to A'_E$ such that the following diagram commutes:
$$\xymatrix{\K\ \ar@{^{(}->}[r]\ar[d]^{Id_\K}   & A_E\ar@{->>}[r]^{\pi}\ar[d]    &A_B\ar[d]^{Id_{A_B}}\\
            \K\ \ar@{^{(}->}[r]                 & A'_E\ar@{->>}[r]^{\pi'}  &A_B.}
$$
\end{rem}
In a practical situation, this theorem might be of no real help for classification. However, it gives a good intuition on how Lie algebroid extensions are made.

\subsection{Isotropies}
For a general extension $\K\hookrightarrow A_E \twoheadrightarrow A_B$ there is no simple relation
between the isotropy Lie algebras $\g^B_{p(x)},$ $\g^E_x$ and $\g^\K_x$ for any $x\in E$. In
particular, these \emph{do not} fit into an extension of Lie algebras in general, as shows example
\ref{infactions}. In the next proposition we summarize a few general properties.

\begin{prop}\label{isotropies}
 Let $\K\hookrightarrow A_E \twoheadrightarrow A_B$ be a Lie algebroid extension. Then the following properties hold:
\begin{enumerate}[i{)}]
\item    $\pi(\g^E_x)   \subset \g^B_{y}$;
\item   $\g^\K_x              =    \g^E_x\cap \K_x$;
\item    $ \pi^{-1}_x(\g^B_y)    =    (\sharp_E)^{-1}_x(\Ver_x)$;
  \end{enumerate}
for any $x\in E$, $y=p(x)$.
\end{prop}

\begin{defn}\label{cleanextensions} We will say that an extension is \textbf{clean} provided the restriction of $\pi$ to $\g^E_x$ surjects onto $\g^B_{p(x)}$ for any $x\in E$.
\end{defn}
As easily seen, in the case of a clean extension the restriction of $\pi$ to $\g^E_x$
induces an extension of Lie algebras:
$$\g^\K_x\hookrightarrow \g^E_x \twoheadrightarrow \g^B_{p(x)}.$$

Moreover, clean extensions enjoy the following property: each  orbit $\O_E$ of $A_E$ fibers over an
orbit $\O_B$ of $A_B$, the typical fiber having as connected components orbits of $\K$ (this
follows from $(i)$ and $(iii)$ in Proposition \ref{isotropies}).

Once fixed an orbit $\O_E$ of $A_E$, one can always choose a Ehresmann
connection such that $h(\g^B_{y})\subset \g^E_x$ for any $y\in B, x\in p^{-1}(y)\cap\O_E$. Note
however that this might not be possible for all $e\in E$ in general.

For such a connection, the restriction of the curvature $2$-form $\omega$ to sections of the isotropy Lie algebra bundle $\g^B_{|\O_B}$ takes values in $\g^\K_{|\O_E}$.


\subsection{Semi-direct products}
\begin{defn} An Lie algebroid extension will be said of \textbf{semi-direct product type} provided it admits a connection whose curvature $2$-form vanishes: $\omega=0$. In particular $\D$ has no curvature as well by \eqref{curvatures}. Once fixed such a connection, we shall write $A_E=A_B\ltimes \K$.
\end{defn}
Since the curvature of $\D$ vanishes, the holonomy along an $A_B$-path depends only on its
$A_B$-homotopy class. This follows from the fact that $\D:\Gamma(A_B)\to\Der(\K)$ is a Lie
algebra morphism and can be proved using arguments similar to those in the proof of proposition
\ref{hgeom} (see also equation \eqref{hflows}).

\subsection{Abelian Extensions}\label{abelian}
\begin{defn} An \textbf{abelian extension} is an extension whose kernel $\K$ is a bundle of abelian Lie algebras over $E$.
\end{defn}
 If one fixes a connection, because of condition \eqref{curvatures}, we have an action of $A_B$ on $E\to B$, plus a representation of the associated action algebroid $A_B\ltimes E$ on $\K$. Note however that $A_B\ltimes \K$ and
$A_E$ may differ in general since the curvature $2$-form $\omega$ might not vanish. In practice,
this means that one gets the bracket on $A_B\ltimes \K$ by the one on $A_E$ by simply deleting the
term involving $\omega$ in \eqref{splitbrackets3}.

In the case $E=B$, then $\K$ is a representation of $A_B$; the curvature $2$-form is
submitted to the only condition (\ref{closedcurvatureform}) which exactly means that it is a closed
form on $A_B$ with values in the representation $\K$. In the Theorem \ref{classification} the
bracket $[\ ,\ ]_{\K}$ vanishes, so it is easily seen that only the cohomology class of $\omega$
determines the extension up to equivalence. In particular, it vanishes if and only if the
extension is of semi-direct product type.

\begin{defn} A \textbf{central extension} is an abelian extension with $E=B$, and $\K$ the trivial bundle of abelian Lie algebras over $B$, $\K=\mathbb{R}^n\times B$ on which $A_B$ acts trivially.
\end{defn}

To be more precise, the condition that $A_B$ acts trivially means that $\D_\alpha
\kappa=\Lie_{\sharp_B(\alpha)}\kappa$ where we see a section of $\K$ as an application $\kappa:B\to
\mathbb{R}^n$. 

The only condition on the curvature form $\omega$ in order to define an extension is to be
$\d_A$-closed. Therefore central extensions on $A_B$ are in one to one correspondence with
closed $2$-form with values in $\mathbb{R}^n$. 

We shall refer to \cite{CZ} for a nice account of such constructions in the context of Poisson and Jacobi manifolds.

\subsection{Examples}\label{examples}
\begin{ex}\emph{Extensions of Lie algebras.}
When both $E$ and $B$ are reduced to a point, $E=B=\{*\}$, both $A_E$ and $A_B$ are Lie algebras,
and we have an extension of Lie algebras. In that case, the notions of abelian and central
extensions coincide with the usual ones.
\end{ex}
\begin{ex}\emph{Submersions.}\label{ex:submersions}
A submersion $p:E\to B$ realizes $TE$ as an extension of $TB$ by the vertical bundle:
$\Ver\hookrightarrow TE \twoheadrightarrow TB$.
\end{ex}

\begin{ex}\emph{Infinitesimal Actions.}\label{infactions}
Consider an infinitesimal action of a Lie algebra $\g$ on a smooth manifold $M$ that is a Lie algebra morphism $\g\to \X(M)$; we denote $\g\ltimes M$ the corresponding action algebroid. Then one gets an extension by taking $E:=M$, $A_E:=\g\ltimes M$ and $A_B=\g$ seen as an algebroid over a point $B=\{*\}$. In that case, the kernel is trivial ${\K}=E\times\{0\}$ and there is a canonically defined Ehresmann connection. In fact one can \emph{define} an infinitesimal action as an extension of a Lie algebra with trivial kernel.

More generally, given an algebroid $A_B\to B$ and an action of $A_B$ on a fibration $E\to B$, the associated action algebroid  $A_B\ltimes E\to E$ is an extension of $A_B$ with vanishing kernel: $\K=E\times\{0\}$. Reciprocally, an extension of an algebroid $A_B$ with vanishing kernel $\K$ is naturally associated to an action of $A_B$ on $E$.

In this situation, one can also form the semi-direct product $A_B \ltimes \Ver\to E$, which is an extension of $A_B$ by $\Ver:=\ker p_*\subset TE$. See \cite{MM} for more details.
\end{ex}

\begin{ex}\emph{Pull-back Lie algebroids.}
Given a submersion $E\to B$ and an algebroid $A_B\to B$ over $B$, the pull-back algebroid $p!A_B$ is defined by:
$$\pi!A_B:=\bigl\{(X,\alpha)\in TE\times A_B/p_*X=\sharp_B\alpha\bigr\}.$$
Clearly, the projection on the second factor $\pi:p!A_B\to A_B$ is a surjective algebroid morphism 
with kernel $\K=\Ver$, where $\Ver\subset TE$ denotes the vertical bundle: $\Ver:=\ker p_*$. A
 connection amounts to a partially defined connection on $\im\sharp_B$, one can for
instance choose a usual Ehresman connection on $E\to B$ and compose it with $\sharp_B$.
\end{ex}

\begin{ex}\emph{Fibered Lie algebroids.}
 In the case of an extension $\K\hookrightarrow A_E \twoheadrightarrow A_B$ where $A_B=\{0\}\times B$, $A_E$ coincides with $\K$ which stands as a fibered Lie algebroid over $p:E\to B$.
\end{ex}

\begin{ex}\emph{Extension by bundles of Lie algebras.}
In the case of an extension $ \K\hookrightarrow A_E \twoheadrightarrow A_B$ where $E$ and $B$ do
coincide: $E=B$, it is easily seen that $\K$ has vanishing anchor so it is a bundle of Lie
algebras.
\end{ex}

\begin{ex}\emph{The Atiyah exact sequence.}\label{ex:Atiyah}
A transitive Lie algebroid $A\to B$ naturally fits into an extension of $A_B=TB$. Here, $B=E$ and the projection $\pi$ is given by the anchor $\pi=\sharp$. Then $\K:=\ker \sharp$ is a bundle of Lie algebras. An Ehresmann connection is then just a section of the anchor, and the corresponding curvature $2$-form is an element $\omega\in\Omega^2(B)\otimes\Gamma(\ker\sharp)$.

A regular Lie algebroid $A\to B$ with foliation $F\subset TB$ can be treated similarly as an extension of $F$ by a bundle of Lie algebras.
\end{ex}

\begin{ex}\emph{Neighborhood of an orbit.}\label{neighborhood}
The following is a straightforward adaptation of the construction in \cite{Vor}.

Consider an algebroid $A_E\to E$, and let $B$ denote one of its orbits. Assume $B$ admits a
tubular neighborhood so that, up to a shrinking of $E$, one is given a smooth submersion $p:E\to
B$, as usual we denote $\Ver:=\ker \ p_*$.

Since along $B$ we have $T_B E=\Ver_{|B}\oplus \im \sharp_E$, the following still holds in a whole neighborhood of $B$: $TE=\Ver+\im\ \sharp_E$. Thus, shrinking $E$ if necessary, one obtains one obtains a surjective
bundle map $A_E\to TB$ by compositing $p_*$ with $\sharp_E$. It is clearly an algebroid morphism
since both $p_*$ and $\sharp_E$ are. In that case, $\K$ is just the transverse Lie algebroid for each
fiber of the projection.
\end{ex}

\begin{ex}\emph{Extensions of tangent bundles}
Here is a rather general procedure to obtain an extension of a tangent bundle $A_B=TB$, directly inspired by the so-called Yang-Mills-Higgs settings for coupling Dirac structures, see for instance \cite{Wa}, \cite{BrFe} and references therein.

Consider a $G$-principal bundle $P\to B$, and an inner action of a group $G$ on an algebroid $A_F\to F$. By definition, this means that we are given an action $\rho:G\to \Aut(A_F)$ of $G$ on $A_F$  by Lie algebroid automorphisms, and a $G$-equivariant algebroid morphism:
$$\psi:\g\ltimes F\to A_F,$$
such that the derivation $\ad_{\psi(\xi)}^{A_F}$ has infinitesimal generator $\frac{d}{dt}\rho_{exp\ t\xi}$.

Then we can build an extension of $TB$ as follows: fix a principal connection $\theta$ on $P$, then we can use parallel transport to obtain a derivation of the fibered Lie algebroid $\K:=P\times_G A_F$:
$$\D:\X(B)\to \Der{(\K)}.$$
Define now $\omega\in\Omega^2(B)\otimes\Gamma(\K)$ by the following formula:
$$\omega(X,Y):=(P\times_G\psi)\bigl(\Curv_\theta(X,Y)\bigr), $$
where $\Curv_\theta:\Omega^2(B)\otimes \Gamma( P\ltimes_G \g)$ denotes the curvature of $\theta$ as a principal connection, and $P\ltimes_G\psi$ denotes the fibered map:
$$P\times_G\psi:P\times_G\g\to P\times_G A_F=\K.$$

It is easily checked that $(\D,\omega)$ induces an extension of $TB$ by $\K$ whose isomorphism
class is independent of the choice of $\theta$.

In such a situation, it is possible to work out ``by hands'' a concrete description of the integrating groupoid, involving only $\G(A_F)$ and bundles associated to $P$. This will be explained in a separate work \cite{BrFe2}
\end{ex}

\section{Cohomology}\label{Cohomology}
In this section, we describe a spectral sequence converging to the cohomology of an
extension. By choosing a connection, we obtain formulas for the successive coboundary operators, which we believe makes it easier to understand. A short example is then discussed.

\subsection{The spectral sequence.}
Given an extension $\K\hookrightarrow A_E \twoheadrightarrow A_B$, there is a natural filtration of the complex of cochains $C^k:=\Omega^k(A_E)$:
$$C^k=:F_0 C^{k}\supset \dots\supset F_qC^{k}\supset F_{q+1}C^{k}\supset\dots\supset F_{k+1}C^{k}:=\{0\},$$
  where the generic generic term is defined as follows:
$$F_q C^k:=\bigl\{\alpha\in\Omega^k(A_E), i_v\alpha=0\ \forall v\in\Gamma(\Lambda^{k-q+1}\K)\bigr\}.$$
Since this filtration is bounded, one gets a spectral sequence converging to $H^\bullet(A_E)$ by
the usual construction. The only thing one actually has to check is that $\d_{A_E}(F_{q}C^k)\subset
F_q C^{k+1}$ which we will leave to the reader (it also can be seen as a consequence of next
proposition).
\begin{rem} We emphasize the change of notations here: up to now, we used Greek letters to denote sections of the different algebroids and Latin ones for vector fields. In this section, Greek letters will be used to denote $k$-forms on Lie algebroids while latine capitals will be used to refer to multi-section of algebroids, as we think of them as generalized multi-vector fields.
\end{rem}

Let us choose an Ehresmann connection (fixed once for all) and recall the identification $\Gamma(A_E)=C^\infty(E)\otimes\Gamma(A_B)\oplus\Gamma(\K)$ induced by the splitting $A_E=\K\oplus H$. Dually one has $\Omega(A_E)=\Omega(A_B)\!\otimes\! C^\infty(E) \oplus \Omega(\K)$. Extending this principle to multi-forms on $A_E$, one gets an isomorphism:
\begin{equation}\label{ident}\Omega^k (A_E)=\bigoplus_{p+q=k} \Omega^p(A_B)\otimes \Omega^q(\K),
\end{equation}
where the projection $\theta^{p,q}$ of some $\Theta\in \Omega^k(A_E)$ on $\Omega^p(A_B)\otimes \Omega^q(\K)$ is given by the following formula:
\begin{equation}\label{splitkforms}\big\langle\theta^{p,q}(v_1\dots v_p),w_1\dots w_q\big\rangle=\big\langle\Theta,h(v_1)\dots h( v_p),w_1\dots w_q\big\rangle,
\end{equation}
where $v_i\in\gamma(A_B),w_i\in\Gamma(\K)$. The reason why this identification is particularly suiting for our purposes is that the filtration $F_q C^k$ comes as successive truncations:
$$F_q C^{k}=\bigoplus_{q\leq p\leq k}\Omega^p(A_B)\otimes \Omega^{k-p}(\K).$$
In particular we have $E_0^{p,q}:=F_q C^{p+q}/F_{q+1} C^{p+q}=\Omega^p(A_B)\otimes \Omega^q(\K)$. In next proposition, we decompose the co-differential operator $d_{A_E}$ under the identification (\ref{ident}), this gives explicit formulas for the successive coboundary operators involved in the spectral sequence.

\begin{prop}\label{splitdifferential}
Let $\K\hookrightarrow A_E \twoheadrightarrow A_B$ be an extension and $A_E=\K\oplus H$ a fixed
Ehresmann connection. Then, under the identification (\ref{ident}) the coboundary
operator  $\d_{A_E}$ on $\Omega^\bullet (A_E^*)$ decomposes as a sum
$\d_{A_E}=\delta_{0,1}+\delta_{1,0}+\delta_{2,-1}$ where $\delta_{i,j}$ have bi-degree $(i,j)$. The restrictions of these operators to $\Omega^p(A_B)\otimes\Omega^q(\K)$ are given by the following formulas:
$$(\delta_{0,1}\theta)(v_1\dots v_p)=(-1)^p \d_{\K}(\theta(v_1\dots v_p)),$$
which is a coboundary operator naturally extending $d_{\K}$ to $\Omega^p(A_B)\otimes\Omega^q(\K)$;
\begin{multline}
(\delta_{1,0}\theta)(v_0\dots v_p)
=\sum_{i=0\dots p}(-1)^{i}\D_{v_i}(\theta(v_1\dots\widehat{v_i}\dots v_p))\\
+\sum_{0\leq i<j\leq n}(-1)^{i+j}\theta([v_i,v_j]_{A_B},v_1\dots\widehat{v_i}\dots\widehat{v_j}\dots v_p),
  \end{multline}
which is the covariant derivative $\partial_H$ associated to the Ehresmann connection as described in
Section \ref{Ehresmann};
$$(\delta_{2,-1}\theta)(v_0\dots v_{p+1})=(-1)^{p+1}\sum_{0\leq i<j\leq p+1}(-1)^{i+j} i_{\omega(v_i,v_j)}\theta(v_0\dots \widehat{v_i}\dots\widehat{v_j}\dots v_{p+1}),$$
where $\omega\in\Omega^2(A_B)\otimes\Gamma(\K)$ is the curvature $2$-form of to the connection.
\end{prop}

\begin{proof}
Let us fix $\theta^{p,q}\in\Omega^p(A_B)\otimes\Omega^q(\K)$ and denote
$\Theta\in\Omega^{p+q}(A_E)$ the corresponding $p+q$-form on $A_E$ through the isomorphism
(\ref{ident}). By construction, $\Theta$ has the following property: for any $v\in\Gamma(\Lambda^s
A_B)$ and $v\in\Gamma(\Lambda^t A_B)$ with $s+t=p+q$, the evaluation $\langle\Theta,h(v)\wedge
v\rangle$ vanishes whenever $s>p$ or $t>q$.

The different components in $\Omega^{s}(A_B)\otimes\Omega^t(\K)$, where $s+t=p+q+1$, of $d_{A_E}\Theta$ are obtained using  \eqref{splitkforms} so we just have to compute $\langle \d_{A_E}\Theta,v\wedge w\rangle$ for any $v\in\Gamma(\Lambda^s A_B)$, $w\in\Gamma(\Lambda^t \K)$, then relate it to $\theta^{p,q}$.

Let us first express $\delta_{0,1}$. We choose sections $w_0\dots w_q$ of $\K$ and  $v_1\dots v_p$ of $A_B$ and compute:
\begin{multline*}
 \big\langle \d_{A_E}\Theta,  h(v_1)\dots  h(v_p), w_0\dots w_q\big\rangle\\
=
\sum_{i=1\dots p}(-1)^{i+1}\Lie_{ \sharp_E  h(v_i)}
                      \big\langle \Theta,  h(v_1)\dots\widehat{  h(v_i)}\dots  h(v_p),
                          w_0\dots w_q\big\rangle \\
    +\sum_{1\leq i<j\leq p}(-1)^{i+j}
            \big\langle \Theta,[  h(v_i),  h(v_j)]_{A_E},
               h(v_1)\dots\widehat{  h(v_i)}\dots\widehat{  h(v_j)}\dots  h(v_p),
              w_0\dots w_q\big\rangle\\
    +(-1)^{p}\sum_{i=0\dots q}(-1)^{i}\Lie_{ \sharp_E w_i}
           \Theta(  h(v_1)\dots  h(v_p),
              w_0\dots\widehat{ w_i}\dots w_q)\\
    +\sum_{0\leq i<j\leq q}(-1)^{i+j}\Theta([ w_i, w_j]_{A_E},
                h(v_1)\dots  h(v_p),
               w_0\dots\widehat{ w_i}\dots\widehat{ w_j}\dots w_q)\\
   +(-1)^{p+1}\sum_{\substack{i=1\dots p\\j=0\dots q}}(-1)^{i+j}
               \Theta([  h(v_i), w_j]_{A_E},
                     h(v_1)\dots\widehat{h(v_i)}\dots  h(v_p),
                    w_0\dots\widehat{w_j}\dots w_q).
\end{multline*}
Here, the two first terms vanish since $\Theta$ is evaluated on $q+1$ sections of $\K$, as well as the last term according to lemma \ref{projectablebrackets}. In the two remaining terms, it is recognized the coboundary operator for the Lie algebroid $\K$ and one obtains the desired expression for $\delta_{0,1}$. Note that one can also write $\delta_{0,1}$ as $\delta_{0,1}=\sum (-1)^p \text{ id}\otimes \d_\K$.

We now compute $\delta_{1,0}$. For any sections  $v_0\dots v_p$ of $A_B$ and $w_1\dots w_q$ of $\K$, we have:
\begin{multline*}
{ \big\langle \d_{A_E}\Theta,h(v_0)\dots  h(v_p),w_1\dots w_q\big\rangle}\\
    =\sum_{i=0\dots p}(-1)^{i}\Lie_{ \sharp_E  h(v_i)}
                      \big\langle\Theta,  h(v_0)\dots\widehat{h(v_i)}\dots  h(v_p),
                          w_1\dots w_q\big\rangle\\
    +\sum_{0\leq i<j\leq p}(-1)^{i+j}
            \big\langle\Theta,[ h(v_i),h(v_j)]_{A_E},
               h(v_0)\dots\widehat{h(v_i)}\dots\widehat{h(v_j)}\dots h(v_p),
              w_1\dots w_q\big\rangle\\
      +(-1)^{p}\sum_{i=1\dots q}(-1)^{i}\Lie_{\sharp_E w_i}
           \big\langle\Theta,  h(v_0)\dots  h(v_p),
              w_1\dots\widehat{ w_i}\dots w_q\big\rangle\\
    +\sum_{1\leq i<j\dots\leq q}(-1)^{i+j}\big\langle\Theta,[ w_i, w_j]_{A_E},
                h(v_0)\dots  h(v_p),
               w_1\dots\widehat{ w_i}\dots\widehat{ w_j}\dots w_q\big\rangle\\
    +(-1)^{p}\sum_{\substack{i=0\dots p\\j=1\dots q}}(-1)^{i+j}
               \big\langle\Theta,[  h(v_i), w_j]_{A_E},
                     h(v_0)\dots\widehat{h(v_i)}\dots h(v_p),
                    w_1\dots\widehat{w_j}\dots w_q\big\rangle.
\end{multline*}
In this sum, the third and fourth terms vanish since $\Theta$ is evaluated on ${p+1}$ sections of
$H$. Then, taking apart the first and last terms, one gets:
\begin{eqnarray*}
&&\sum_{i=0\dots p}(-1)^{i}\Lie_{ \sharp_E  h(v_i)}
                      \big\langle\Theta,h(v_0)\dots\widehat{h(v_i)}\dots h(v_p),
                          w_1\dots w_q)\\
&&+(-1)^{p}\sum_{\substack{i=0\dots p\\j=1\dots q}}(-1)^{i+j}
               \big\langle\Theta,\D_{v_i} w_j,
                     h(v_0)\dots\widehat{h(v_i)}\dots  h(v_p),
                    w_1\dots\widehat{ w_j}\dots w_q\big\rangle\\
&=&\sum_{i=0\dots p}(-1)^{i}\Lie_{ \sharp_E  h(v_i)}
                      \big\langle\theta^{p,q}(v_0\dots\widehat{v_i}\dots v_p),
                         w_1\dots w_q\big\rangle\\
&&+\sum_{\substack{i=0\dots p\\j=1\dots q}}(-1)^{i+j}
               \big\langle\theta^{p,q}(v_0\dots\widehat{v_i}\dots v_p),
                   \D_{v_i} w_j, w_1\dots\widehat{ w_j}\dots w_q)\big\rangle\\
&=&\sum_{i=0\dots p}(-1)^{i}\big\langle\D_{v_i}
                     \ \theta^{p,q}(v_0\dots\widehat{v_i}\dots v_p),
                         w_1\dots w_q\big\rangle.
\end{eqnarray*}
Considering now the second term, we know that the following expression:
$$\big\langle\Theta,\omega(v_i,v_j),
               h(v_0)\dots\widehat{h(v_i)}\dots\widehat{h(v_j)}\dots h(v_p),
              w_1\dots w_q\big\rangle$$
 vanishes because $\Theta$ is evaluated on $q+1$ section of $\K$. Since $\omega$ precisely measures the difference $\omega(v_i,v_j)=[h(v_i),h(v_j)]_{A_E}-h([v_i,v_j]_{A_B})$ one deduces that
\begin{multline*}
\sum_{0\leq i<j\leq p}(-1)^{i+j}
            \big\langle\Theta,[h(v_i),h(v_j)]_{A_E},
               v_0\dots\widehat{  v_i}\dots\widehat{ v_j}\dots  v_p,
              w_1\dots w_q\big\rangle\\
=\sum_{0\leq i<j\leq p}(-1)^{i+j}
            \big\langle\Theta,h([v_i,v_j]_{A_B}),
               v_0\dots\widehat{  v_i}\dots\widehat{ v_j}\dots  v_p,
              w_1\dots w_q\big\rangle\\
=\sum_{0\leq i<j\leq p}(-1)^{i+j}
            \big\langle\theta^{p,q}([v_i,v_j]_{A_B},v_0\dots\widehat{v_i}\dots\widehat{v_j}\dots v_p),
             w_1\dots w_q\big\rangle.
\end{multline*}
So we obtain the announced formula for $\delta_{1,0}$. It appears to coincide with the covariant
derivative $\partial_H$ associated to the connection.

Let us now express $\delta_{2,-1}$. For any sections $ w_1\dots w_{q-1}$ of $\K$ and $ v_0\dots v_{p+1}$ of $  A_B$ we check that:
\begin{multline*}
\big\langle \d_{A_E}\Theta,  h(v_0)\dots  h(v_{p+1}),w_1\dots w_{q-1}\big\rangle\\
=\sum_{i=0\dots p+1}(-1)^i\Lie_{ \sharp_E  h(v_i)}
                      \big\langle\Theta,  h(v_0)\dots\widehat{h(v_i)}\dots  h(v_{p+1}),
                          w_1\dots w_{q-1}\big\rangle\\
    +\sum_{0\leq i<j \leq p+1}(-1)^{i+j}
            \big\langle\Theta,[h(v_i),h(v_j)]_{A_E},
               h(v_0)\dots\widehat{h(v_i)}\dots\widehat{h(v_j)}\dots h(v_p),
             w_1\dots w_{q-1}\big\rangle\\
    +(-1)^{p+1}\sum_{j=1\dots q}(-1)^{i}\Lie_{ \sharp_E w_i}
           \big\langle\Theta,  h(v_0)\dots  h(v_{p+1}),
             w_1 \dots\widehat{ w_i}\dots w_{q+1})\\
    +\sum_{1\leq i<j \leq q-1}(-1)^{i+j}\Theta([ w_i, w_j]_{A_E},
                h(v_0)\dots h(v_{p+1}),
               w_1\dots\widehat{ w_i}\dots\widehat{ w_j}\dots w_{q-1}\big\rangle\\
    +(-1)^{p+1}\sum_{\substack{i=0\dots p+1\\j=1\dots q-1}}(-1)^{i+j}
               \big\langle\Theta,[h(v_i), w_j]_{A_E},
                     h(v_0)\dots\widehat{h(v_i)}\dots  h(v_{p+1}),
                    w_1\dots\widehat{w_j}\dots w_{q-1}\big\rangle.
\end{multline*}
Here, it is easily seen that only the second term does not vanish. To deal with it, one notices that the expression $$\big\langle\Theta,h([v_i,v_j]_{A_B}),
               h(v_0)\dots\widehat{h(v_i)}\dots\widehat{h(v_j)}\dots h(v_p),
             w_1\dots w_{q-1}\big\rangle$$
 vanishes since $\Theta$ is evaluated on $p+1$ section of $H$, and since $\omega$ precisely measures the difference $\omega(v_i,v_j)=[h(v_i),h(v_j)]_{A_E}-h([v_i,v_j]_{A_B})$ one gets the expression for $\delta_{2,-1}$ stated in the proposition.

To conclude one has to check that $\d_A\Theta$ has no other components than those three mentioned
in this proof, but this should be clear by what has been done so far.\end{proof}

In order to keep notations meaningful, in the sequel, we will write respectively $\partial_\K$,
$\partial_H$, and $\partial_\omega$ rather then $\delta_{0,1}, \delta_{1,0}$ and $\delta_{2,-1}$.
Using the bi-graduation and the fact that $\d_{A_E}\circ \d_{A_E}=0$ one easily gets the following
useful formulas.
\begin{prop}\label{differentialproperties} Under the hypothesis of proposition \ref{splitdifferential} and notations above, the following relations holds:
 \begin{enumerate}[i{)}]
  \item $\partial_\K\circ \partial_\K             \hspace{101pt}=0$,
  \item $ \partial_\K\circ\partial_H+\partial_H\circ \partial_\K\hspace{50pt}=0$,
  \item $ \partial_\K\circ \partial_\omega \hspace{2pt}+\partial_H\circ\partial_H + \partial_\omega \circ \partial_\K\hspace{1pt}=0$,
  \item $\hspace{50pt}\partial_H \circ \partial_\omega \hspace{3pt}+\partial_\omega \circ \partial_H=0 $,
  \item $\hspace{103pt}\partial_\omega\circ \partial_\omega\hspace{2pt}=0$.
 \end{enumerate}
\end{prop}

\subsection{Interpretation}
Let us now explain how these results can give some sort of geometric approach to this cohomology. Here, we will carry out the usual construction of a spectral sequence associated to a filtration, using standard notations.

We explained already why:
$$E_0^{p,q}:=F_q^{p,q}\Omega(A_E)^{p+q}/F_{q+1}^{p,q}\Omega^{p+q}(A_E)=\Omega^p(A_B)\otimes \Omega^q(\K),$$
and we have $\d_0=\text{id}\otimes \d_\K$ according to the Proposition \ref{splitdifferential}.

Now the first sheet $E_1^{p,q}$ of the spectral sequence can be thought of as forms on $A_B$ with values in the cohomology $H^*(\K)$ of $\K$ because:
$$E_1^{p,q}:=\Omega^p(A_B)\otimes Z^q(\K) / \Omega^p(A_B)\otimes B^q(\K),$$
where we denoted:
\begin{align*}
Z^q(\K)&:=\bigl\{\mu \in\Omega^q(\K): \d_\K \mu=0\bigr\},\\ 
B^q(\K)&:=\bigl\{\mu\in\Omega^q(\K): \mu=\d_\K\nu\bigr\}.
\end{align*}
 The coboundary operator $\d_1$ on $E_1^{p,q}$ is the quotient
of the map $\partial_H$. Let us see why $\d_1\circ \d_1=0$: an element in $E_1^{p,q}$ is
represented by some $\theta\in \Omega^p(_B)\otimes Z^q(\K)$ thus, using $(iii)$ in proposition
\ref{differentialproperties}, one has:
$$-\partial_H\circ\partial_H\, \theta=\partial_\K\circ \partial_\omega\,\theta+ \partial_\omega \circ \partial_\K\,\theta=\partial_\K\circ \partial_\omega\,  \theta,$$
which indeed represents a trivial element in $E^{p+1,q}_1$.

Now, in order to understand the second sheet $E_2^{p,q}$, one shall think of elements in $H^*(\K)$ as sections of a vector bundle over $B$. This makes sense since $H^*(\K)$ carries  a natural structure of $C^\infty(B)$-module (this because $\d_\K$ is $C^\infty(B)$-linear). From this point of view $\partial_H$ induces a representation of $A_B$ on $H^*(\K)$ by setting $D_X ([\theta]_\K):=[\partial_H \theta (X)]_\K$. The equality $(ii)$ in proposition \ref{differentialproperties} ensures that $D$ it is well defined, and the last computation with $p=0$ that it is indeed a representation. Thus the second sheet $E_2^{p,q}$ is as close as one can get to a cohomology of $A_B$ with values in the representation  $H^*(\K)$.

The last sheet $E_3^{p,q}$ specifically involves $\partial_\omega$,
it will be the crucial step distinguishing the cohomologies of two extensions with same derivative $\D$ but distinct curvatures forms $\omega, \omega'$.

We stress the fact that the spaces we are dealing with here can be very singular. In some examples though, this construction allows to carry the computation of the cohomology through. Of course, this assumes a good knowledge of the cohomologies of $A_B$ and $\K$.

\begin{ex} Applying these results to an action Lie algebroid (see Example \ref{infactions}), we get the well-known fact that the cohomology of $g\ltimes M$  is the same as the Chevalley Eilenberg cohomology of $\g$ with values in the $\g$-module $C^\infty(M)$.
\end{ex}

\begin{ex} In the case of a fibration $E\to B$ (see Example \ref{ex:submersions}), one obtains the usual Leray Serre spectral sequence, involving vertical de Rham cohomology on $E$ and the de Rham cohomology of the base.
\end{ex}
\begin{ex} In the case of Poisson cohomology relative to a symplectic leaf, which is a particular case of cohomology of a transitive Lie algebroid, one obtains the same formula as in \cite{GL}.
\end{ex}

\subsection{An example of computation}
Let us present a simple but non-trivial example where the computation of the cohomology can be carried out.

We consider a Poisson structure on $E:=\RR^+\times S^2$ whose symplectic leaves are $\{r\}\times S^2$ endowed with the symplectic form $f\omega_0$, where  $f$ is a nowhere vanishing function of $r\in \RR^+$ and $\omega_0$ the usual symplectic form on $S^2$. Recall that these settings allow the simplest examples of non-integrable Lie algebroids. The associated Poisson cohomology of $E$ was computed in \cite{Gi}, we shall also mention \cite{X} for earlier work on regular Poisson structures. We present here an approach using extensions.

We now set $A_E=T^*E=T^*\RR^+\times T^*S^2$, and we let $p:\RR^+\times S^2\to S^2$ be projection on the second factor and $\pi:T^*E\to TS^2$ be the composition of $\sharp_E$ with $p_*$. Thus $T^*E$ is an extension of $A_B=TS^2$ by $\K=T^*\RR^+\times S^2$; note that forms on $\K$ are $p$-vertical multi-vector fields.

One defines an Ehresmann connection by letting: $h(X):=f i_X\omega_0$. Thus wee see that the identification $T^*E=T^*\RR^+\times T^*S^2$ also corresponds with the decomposition $A_E=\K\oplus H$. The curvature form $\omega\in \Omega^2(B)\otimes\Gamma(\K)$ is easily computed:
$$\omega(X,Y)=\omega_0(X,Y)\otimes f'\d r.$$

Let us now compute the associated cohomology. The kernel $\K$ is an abelian algebroid (with trivial anchor and bracket) so that $d_\K=0$. This implies $E_0^{p,q}=E_1^{p,q}=\Omega^p(TB)\otimes \Omega^q( \K)$.

We now compute the coboundary operator $d_1$ on $E_1^{0,0}=C^\infty(E)$ and $E_1^{0,1}=\Omega^1(\K)$. Using formulas in proposition \ref{splitdifferential} we get:
$$\begin{array}{ccc}
   (\partial_H g)(X)&=&\Lie_X g, \\
   (\partial_H\,g.\partial r)(X)&=&(\Lie_X g)\partial r,\\
  \end{array}$$
for any $g\in C^\infty(E)$. So we see that $\partial_H(g)=0$ or $\partial_H(g\partial r)=0$ if and only if $g$ only depends on $r$. It follows that $E_2^{0,0}$ and $E_2^{0,1}$ can be identified to $C^\infty(\RR^+)$.

The explicit formula for $d_1$ on $E^{1,0}_1=\Omega^1(TB)\otimes\C^\infty(E)$ is given by:
$$ (\partial_H \alpha)(X,Y)=\Lie_X \alpha(Y)-\Lie_Y \alpha(X)-\alpha([X,Y]),\ \ (\alpha \in E^{1,0}_1).$$
Here $\alpha(X),\alpha(Y),$ and $\alpha([X,Y])$ are function on $\RR^+\times S^2$ but since an element $\alpha$ of $E^{1,0}_1$ is the same as a smooth family of $1$-forms on $S^2$ with parameter $r$, we see that $\partial_H\alpha=0$ if and only if it determines a family of closed forms. It can be made into a family of exact $1$-forms because $H^1(S^2)=\{0\}$ and this can be done in a smooth way with respect to $r$; so we conclude that $E^{1,0}_2=\{0\}$. For the same reason, if we use the pairing with $\d r$ to identify $E_1^{1,1}$ with $\Omega^2(B)\otimes C^\infty(E)$, we one get the same formula for $\partial_H$ as above, thus $E^{1,1}_1=\{0\}$ as well.

The same kind of reasoning applies to compute $E_2^{2,0}$: we know that integration of $2$-forms over $S^2$ induces an isomorphism $H^2(S^2)\to \RR$. A $\partial_H$-closed element in $E_1^{2,0}=\Omega^2(B)\otimes \C^\infty(E)$ being the same thing as a family of closed $2$-forms on $S^2$ with parameters in $\RR^+$, one sees that integration over $S^2$ with fixed parameter yields an isomorphism $E_2^{2,0}\simeq C^\infty(\RR^+)$ given by:
$$\text{[}g.\omega_0\text{]}\longmapsto \bigr(r\mapsto \int_{S^2} g(r,\ )\omega_0\bigl).$$
For the same reasons, one gets has isomorphism $E^{2,1}_{2}\simeq C^\infty(\RR^+)$ via $$\text{[}\omega_0\otimes g\partial r\text{]}\longmapsto \bigl(r \mapsto \int_{S^2} g(r,\ )\omega_0\bigr).$$

These identifications are convenient for our purposes because, as easily checked, the explicit
formula for $\d_2:=[\d_\omega]:E_2^{0,1}\to E_2^{2,0}$ is just multiplication by $f'$. Since $\d_2$
vanishes on $E_2^{p,q}$ for $(p,q)\neq(2,-1)$ the computation is easily concluded. We finally obtain the
Poisson cohomology of $T^*E$ as follows:
$$\begin{array}{rcl}
    H^0(T^*E)&=&C^\infty(\RR^+),\\
    H^1(T^*E)&=&\{g\in C^\infty(\RR^+), f'g=0\},\\
    H^2(T^*E)&=&C^\infty(\RR^+)/f'.C^\infty(\RR^+),\\
    H^3(T^*E)&=&C^\infty(\RR^+).
  \end{array}
$$

\section{Integration}\label{integration}

We will here take advantage of a connection in order to describe ${A_E}$-paths and ${A_E}$-homotopies, obtaining an alternative description of $\G({A_E})$. From now on, the connection will always be assumed to be {\bf{complete}}.

\subsection{\texorpdfstring{Splitting ${A_E}$-paths and ${A_E}$-homotopies}{Splitting paths and homotopies}}
\label{splitting:paths}
Since the essential feature of a Ehresmann connection is to split $A_E$ into a direct sum
\begin{equation*}A_E=\K\oplus H,\end{equation*}
it is natural, in order to describe $\G(A_E)$, to try to decompose $A_E$-paths accordingly. However, for a $A_E$-path $a:I\to A_E$ covering $\gamma:I\to E$, the corresponding decomposition is the following:
$$a(t)=a_V(t)+h(a_B(t)),$$
where $a_B:=\pi\circ a$ and $a_V:I\to \K$. This is not satisfying because none of the two terms above is an $A_E$-path.

Thus, it is better to consider the couple $(\dot{\gamma}_B,a_K)$ where
\begin{equation}\label{atilde}a_K(t):=\Phi^{\gamma_B}_{0,t}\circ a_V(t).
\end{equation}
Here $\Phi^{\gamma_B}_{t,0}:\K_{|E_{\gamma_B(0)}}\mapsto \K_{|E_{\gamma_B(t)}}$ denotes the parallel transport along $a_B$. Indeed, it is easily checked that $a_K$ is a $A_E$-path over  $\gamma_K(t)=\phi^{\gamma_B}_{0,t}\circ\gamma(t)$ lying in $E_{\gamma_B(0)}$. Moreover the correspondence $a\leftrightarrow(a_B,a_K)$ is clearly $1$-$1$ and if we assume the connection to be complete, we see that we have built an homeomorphism 
$$P(A_E)\to P(A_B)\ltimes_B P(\K),$$ 
where $P(A_B)\ltimes_B P(\K)$ denotes couples $(a_B,a_K)\in P(A_B)\times P(\K)$ such that $a_K$ is a $\K$-path lying over $\gamma_B(0)$. Recall that $P(A_E)$ enjoys a natural structure of a Banach space (see \cite{CrFe2} for more details)

Now if we look at how concatenation in $P(A_E)$ behaves through this homeomorphism, we see that the concatenation $v\cdot u$ of two $A_E$-paths $u,v\in P(A_E)$ is sent to $(v_B\cdot u_B,\Phi_{u_B}^{-1}(v_K)\cdot u_K)$ where we have set:
\begin{equation}\label{eq:holonomy:paths}\Phi_{u_B}^{-1}(v_K)(t):=\Phi^{u_B}_{0,1}\circ v_K(t).
\end{equation}

Thus we have proved the following:

\begin{prop}\label{splitpath}{(Splitting ${A_E}$-paths).}
Let $\K\hookrightarrow A_E \twoheadrightarrow A_B$ be a Lie algebroid extension, and fix an Ehresmann connection which is assumed to be complete. Then there is a homeomorphism:
$$\begin{array}{ccc}
 P({A_E})   &\to& P({A_B})\ltimes_B P({\K})\\
    a      &\mapsto&({a_B},{a_K}).
\end{array}$$
under which the concatenation in $P(A_E)$ writes:
\begin{equation}\label{concpath}(a^2_B,a^2_K)\cdot(a^1_B,a^1_K)=(a^2_B\cdot a^1_B,\Phi_{a^1_B}^{-1}(a^2_K)\cdot a^1_K).
\end{equation}
\end{prop}

\begin{rem}
 Let us explain the picture we have in mind: given an algebroid $A\to M$, we can think of $P({A})$ as a ``groupoid'' over $M$: source and target are defined the obvious way and concatenation plays the role of multiplication. One also has the notions of ``units'' and of ``inverses'': $1_m(t):=0_m$, and $a^{-1}(t):=-a(1-t)$ (note however that $a\cdot a^{-1}\neq 1_{\gamma(0)}$). Morphisms of ``groupoids'' should only commute with source, target and multiplication, and sent units on units, the most obvious way to obtain such a morphism being induced by a algebroid morphism (rather than introducing useless vocabulary, we just stick to ``groupoids'' without further precisions).

Generalizing the notion of action introduced in Appendix \ref{fiberedliegroupoids} to ``groupoids'', we can use the holonomy $\Phi:P({A_B})\mapsto \Gau(P(\K))$ to get an action of $P({A_B})\tto B$ on the fibered ``groupoid'' $P(\K)\to B$, and then build the corresponding action ``groupoid'' $P({A_B})\ltimes P(\K)\tto E$. Then we see that the  formula for composition is exactly the one given in the Proposition \ref{splitpath}. Thus we can think of $P(A_E)$ as a groupoid associated to an action on a fibered groupoid.
\end{rem}

Because of the presence of curvature, there is no action of the groupoid $\G(A_B)$ neither on $P(\K)$, nor on $\G(\K)$. However, it is still possible to describe ${A_E}$-homotopies through this homeomorphism in a reasonable way, as stated in next proposition.

Let us before precise a natural procedure to extend a ${A_E}$-path $a$ into a time-dependant
section of ${A_E}$ we will always implicitly use: first extend arbitrarily ${a_B}(t)$ into
$\alpha_B(t)\in \Gamma(A_B)$ and $a_V$  into a time-dependant
section $\alpha_V$ of $\K\subset {A_E}$. Clearly, $\alpha:=h(\alpha_B(t))+\alpha_V$ extends $a$.
Moreover, $\alpha_K(t):=(\Phi^{a_B}_{0,t})_*(\alpha_V(t))$ extends the $\K$-path $a_K$ defined by
(\ref{atilde}).

\begin{prop}\label{prop:splithomotopy}
Let $a^\epsilon$ be a smooth family  of $A_E$-paths and $(a_B^\epsilon,a_K^\epsilon)$ the
corresponding family in $P({A_B})\ltimes P(\K)$ given by last proposition.

Then $a^\epsilon$ defines a $A_E$-homotopy  if and only $a_B^\epsilon$ is a ${A_B}$-homotopy such
that the unique solution $\mu_K\in\Gamma(\K)$ of the evolution equation:
\begin{equation}\label{splitevolution}
 \frac{d\alpha_K}{d\epsilon}-\frac{d\mu_K}{dt}=\bigl[\alpha_K,\mu_K\bigr]_{\K},
\end{equation}
with initial condition $\mu_K^0(\epsilon)=0$ satisfies:
\begin{equation}\label{splitcondition}{\mu_K^1(\epsilon)}_{\gamma_K^\epsilon(1)}=\int_0^1 (\Phi^{a_B^\epsilon}_{0,s})_*\Bigl(\omega(a_B,b_B)_{s,\epsilon}\Bigr)ds, \ \forall\epsilon\in I,
\end{equation}
Here, $a_B\d t+ b_B\d\epsilon:TI^2\to A_B$ is the Lie algebroid morphism induced by the $A_B$-homotopy $a_B^\epsilon$, and $\alpha_\K$ is extending $a_K$ as explained above.
\end{prop}

\begin{proof}
 We first extend the smooth family of $A_E$-paths $a^\epsilon$ into a smooth family of time-dependant sections $\alpha^\epsilon$
 as specified above: $\alpha=h(\alpha_B)+\alpha_V$. It is easily checked that the solution $\beta$ of the evolution equation \eqref{evolution} is
 $\pi$-projectable provided $\alpha$ and the initial condition are. Thus, we can write the solution as  $\beta=h(\beta_B)+\beta_V$ and use
 the formula for brackets \eqref{splitbrackets1} \eqref{splitbrackets2} \eqref{splitbrackets3}.  We see that the evolution equation \eqref{evolution} in ${A_E}$
 is equivalent to the following two evolution equations
\begin{align}
\frac{d\alpha_B}{d\epsilon}-\frac{d\beta_B}{dt}=&[\alpha_B,\beta_B]_{{A_B}}\label{splithomotopy1},\\
\D_{\beta_B}(\alpha_V)+\frac{d\alpha_V}{d\epsilon}-\D_{\alpha_B}(\beta_V)+\frac{d\beta_V}{d  t}=&[\alpha_V,\beta_V]_{\K}+\omega(\alpha_B,\beta_B), \label{splithomotopy2}
\end{align}
with initial conditions $\beta_B^0(\epsilon)=0$ and $ \beta_V^0(\epsilon)=0$. The homotopy
condition \eqref{hcondition} is clearly equivalent to $\beta_B^1(\epsilon)=0,
\beta_V^1(\epsilon)=0$ and in particular, one gets from (\ref{splithomotopy1}) that
$a_B^{\epsilon}$ is an homotopy in ${A_B}$.

Let us now look at equation (\ref{splithomotopy2}): in order to get an evolution equation
involving $\alpha_K^\epsilon$, we push-forward (\ref{splithomotopy2}) by means of
$\Phi^{a_B^\epsilon}_{0,t}$:
\begin{multline*}(\Phi^{a_B^\epsilon}_{0,t})_*\bigl(\D_{\beta_B}(\alpha_V)+\frac{d\alpha_V}{d\epsilon}\bigr)
 -(\Phi^{a_B^\epsilon}_{0,t})_*\bigl(\D_{\alpha_B}(\beta_V)+\frac{d\beta_V}{d t}\bigr)\\
  =(\Phi^{a_B^\epsilon}_{0,t})_*\bigl([\alpha_V,\beta_V]_{\K}+\omega(\alpha_B,\beta_B)\bigr).
\end{multline*}
Only the first term on the left-hand side of this equality is problematic, it will be taken care
of in the next lemma. This lemma shows that (\ref{splithomotopy2}) holds if and only if the following
holds:
\begin{multline*}
\frac{d}{d\epsilon}(\Phi^{a_B^\epsilon}_{0,t})_*(\alpha_V)-\Bigl[\int_0^t (\Phi^{a_B^\epsilon}_{0,s})_*\bigl(\omega(\alpha_B,\beta_B)_{s,\epsilon}\bigr)ds,(\Phi^{a_B^\epsilon}_{0,t})_*(\alpha_V)\Bigl]_{\K}\\
-\frac{d}{dt}(\Phi^{a_B^\epsilon}_{0,t})_*(\beta_V)
=\bigl[(\Phi^{a_B^\epsilon}_{0,t})_*(\alpha_V),(\Phi^{a_B^\epsilon}_{0,t})_*(\beta_V)\bigr]_{\K}+(\Phi^{a_B^\epsilon}_{0,t})_*(\omega(\alpha_B,\beta_B)).
\end{multline*}
In other words, if we set
$$\left\{\begin{array}{ccl}
\beta_K^t(\epsilon)&:=&(\Phi^{a_B^\epsilon}_{0,t})_*(\beta_V^t(\epsilon)),\\
F_{\epsilon,t}&:=&\displaystyle \int_0^t (\Phi^{a_B^\epsilon}_{0,s})_*\bigl(\omega(\alpha_B,\beta_B)_{s,\epsilon}\bigr)ds,
\end{array}\right.$$
we see that (\ref{splithomotopy2}) holds if and only if $\beta_K+F$ is solution of equation (\ref{splitevolution}):
\begin{equation*}
\frac{d}{dt}(\beta_K+F)-\frac{d{\alpha_\K}}{d\epsilon}+\bigl[\alpha_K,\beta_K+F\bigr]_{\K}=0.
\end{equation*}
 Now observe that the homotopy condition $\beta_V^1(\epsilon)=0$ holds if and only if
 $\beta_K^1(\epsilon)=0$ (here, we voluntarily omit base points). By uniqueness of the solution $\mu_K=\beta_K+F$ of \eqref{splitevolution} we see that $\beta_V^1(\epsilon)=0$ if and only if \eqref{splitcondition} holds.
\end{proof}

\begin{lem}\label{small:lemma} With the same notations as above, one has:
$$(\Phi^{a_B^\epsilon}_{0,t})_*\bigl(\D_{ \beta_B}({\alpha_V})+\frac{d\alpha_V}{d\epsilon}\bigr)=\frac{d}{d\epsilon}(\Phi^{a_B^\epsilon}_{0,t})_*(\alpha_V)-\bigl[F,(\Phi^{a_B^\epsilon}_{0,t})_*(\alpha_V)\bigl]_{\K}.$$
\end{lem}

\begin{proof}
Consider the solution $\sigma$ of the following evolution equation:
\begin{eqnarray} \frac{d\sigma}{dt}-\frac{dh(\alpha_B)}{d\epsilon}+[h(\alpha_B),\sigma]_{A_E}=0\label{sigmaevolution},
\end{eqnarray}
with initial condition $\sigma^0(\epsilon)=0$. We can use the same principle as in the proof of Proposition \ref{hgeom} to show that the flow  $\psi^{\sigma^t}$ of the time-dependant derivation $ad^{A_E}_{\sigma^t}$ satisfies:
\begin{eqnarray}\label{sigmaflow}
\psi^{\sigma^t}_{0,\epsilon}=\Phi^{a_B^0}_{t,0}\circ\Phi^{a_B^\epsilon}_{0,t}.
\end{eqnarray}
Let us now relate $\sigma$ and $h(\beta_B)$: first applying $h(-)$ to (\ref{splithomotopy1}), one gets:
$$\frac{dh(\beta_B)}{dt}-\frac{dh(\alpha_B)}{d\epsilon}+h\bigl([\alpha_B,\beta_B]_{A_B}\bigr)=0. $$
In this equality, one can subtract equation (\ref{sigmaevolution}) to get the following:
\begin{equation}
 \label{curvature:difference}
\frac{d(\sigma-h(\beta_H))}{dt}+\bigl[h(\alpha_B),\sigma-h(\beta_B)\bigr]_{A_E}=\omega(\alpha_B,\beta_B).
\end{equation}
By applying $(\Phi^{a_B^\epsilon}_{0,t})_*$ to this last equation, we obtain the desired relation
between $\sigma$ and $h(\beta_B)$:
\begin{equation}\label{sigmacurv} \frac{d}{dt}(\Phi^{a_B^\epsilon}_{0,t})_*\bigl(\sigma-h(\beta_B)\bigr)=(\Phi^{a_B^\epsilon}_{0,t})_*\bigl(\omega(\alpha_B^{\epsilon,t},\beta_B^{\epsilon,t})\bigr).
\end{equation}
We now have enough material to prove the equality stated by the lemma: decompose the left-hand term as follows:
\begin{multline*}(\Phi^{a_B^\epsilon}_{0,t})_*\Bigl(\D_{\beta_B}({\alpha_V})+\frac{d\alpha_V}{d\epsilon}\Bigr)
   =(\Phi^{a_B^\epsilon}_{0,t})_*\Bigl([\sigma,{\alpha_V}]+\frac{d\alpha_V}{d\epsilon}\Bigr)
          +(\Phi^{a_B^\epsilon}_{0,t})_*\bigl([h(\beta_B)-\sigma,{\alpha_V}]\bigr).
\end{multline*}
In this expression, we first check using \eqref{sigmaflow} that the first term gives $\frac{d}{d\epsilon}(\psi^{\gamma_B^\epsilon}_{0,t})_*(\alpha_V)$:
$$\begin{array}{ccc}
(\Phi^{a_B^\epsilon}_{0,t})_*\bigl([\sigma,{\alpha_V}]+\frac{d\alpha_V}{d\epsilon}\bigr)
&=&(\Phi^{a_B^0}_{0,t})_*\circ(\psi^{\sigma^t}_{0,\epsilon})_*\bigl([\sigma,{\alpha_V}]+\frac{d\alpha_V}{d\epsilon}\bigr)\\
&=&(\Phi^{a_B^0}_{0,t})_*\bigl(\frac{d}{d\epsilon}(\psi^{\sigma^t}_{0,\epsilon})_*(\alpha_V)\bigr)\\
&=&\frac{d}{d\epsilon}(\Phi^{a_B^0}_{0,t})_*\circ(\psi^{\sigma^t}_{0,\epsilon})_*(\alpha_V)\\
&=&\frac{d}{d\epsilon}(\Phi^{a_B^\epsilon}_{0,t})_*(\alpha_V).
\end{array}$$
It is easily integrated relation \eqref{sigmacurv} to show that the second term gives
$$
(\Phi^{a_B^\epsilon}_{0,t})_*\bigl([h(\beta_B)-\sigma,{\alpha_V}]\bigr)=\bigl[F,(\psi^{\gamma_B^\epsilon}_{0,t})_*(\alpha_V)\bigl]_{\K},$$
which completes the proof.\end{proof}

Now that we have translated the homotopy condition to a condition on the split paths, the
theorem below follows easily from Proposition \ref{hgeom}.

\begin{thm}\label{thm-2cocycle} Consider a Lie algebroid extension $ K\hookrightarrow A_E \twoheadrightarrow A_B,$
endowed with a Ehresmann connection $A_E=\K\oplus H$, which we assume to be complete.

Then the topological source simply connected groupoid $\G({A_E})$ integrating ${A_E}$ is naturally
identified with the quotient
$$P(A_B)\ltimes_B \G({\K})/\sim$$
where the equivalence relation is given by: $(a_0,g_0)\sim(a_1,g_1)$ if and only if there exists a
${A_B}$-homotopy $h_B=a \d t+b \d\epsilon:TI^2\mapsto {A_B}$ between $a_0$ and $a_1$ such that:
$$g_1\cdot g_0^{-1}=\partial (h_B,\t(g_0)).$$ Here, $\partial(h_B, x_0)$ is the element in $\G(\K)$
represented by the ${\K}$-path:
\begin{equation}\label{partial:path}\epsilon\rightarrow\int_0^1 (\Phi^{a_B^\epsilon}_{0,s})_*\Bigl( \omega(a,b)_{s,\epsilon}\Bigr)\d s\in {\K}_{{\gamma_K}^\epsilon},
\end{equation}
where $\gamma_K^\epsilon:=\phi^{-1}_{a_B^\epsilon}\circ \phi_{a_B^0}(x_0).$
\end{thm}
\begin{ex}
In the case of central extensions, Theorem \ref{thm-2cocycle} reduces to a result due to Crainic in a Lie algebroid approach to prequantization \cite{Cr}, see also \cite{CZ} in the context of Poisson and Jacobi manifolds. Recall that
 $E=B$ and $\K:=\mathbb{R}^n\times B$ is a bundle of abelian Lie algebras on which $A_B$
acts trivially, while $\omega$ is a $d_A$-closed $2$-form on $A_B$ with values in $\mathbb{R}^n$. 

Then for any $A_B$-path $a_B$ covering $\gamma_B$, the holonomy
$\Phi_{a_B}:\K_{\gamma_B(0)}\to\K_{\gamma_B(1)}$ is trivial:
$\Phi_{a_B}(r,\gamma_B(0))=(r,\gamma_B(1)).$ Moreover $\G(\K)$ is identified with the bundle of abelian Lie groups $B\times \mathbb{R}^n$ by averaging $\K$-paths. This process makes appear a second integral in the formula \eqref{partial:path} and the homotopy condition reads as follows: two elements $(a_0,g_0)$ and $(a_1,g_1)$
in ${\mathbb{R}^n}\times P(A_B)$ are homotopic if and only if there exists a ${A_B}$-homotopy
$\gamma_B:TI^2\mapsto {A_B}$ between $\gamma_0$ and $\gamma_1$ such that:
$$ g_1-g_0=\int_{\gamma_B} \omega.$$

The Theorem \ref{thm-2cocycle} might be seen as a non abelian generalization of this result.
\end{ex}

\begin{rem} There is an obvious way to concatenate two $A_B$-homotopies $h=a\d t+b\d\epsilon, \tilde{h}=\tilde{a}\d t+\tilde{b}\d\epsilon$ such that
$h^{\epsilon=1}={\tilde{h}}^{\epsilon=0}$. This amounts to concatenate two algebroid morphisms $TI^2\to
A_B$ along the edges $\{\epsilon=1\}$ and $\{\epsilon=0\}$ of two copies of $I^2$.

To do this properly, one needs to reparametrize in the $\epsilon$ variable in order to make sure the concatenation is smooth (on the way one sees that it is enough to assume that $a_{\epsilon=1}=\tilde{a}_{|\epsilon=0}$ to be able to concatenate). We leave the details to the reader since the construction is a straightforward adaptation of the concatenation for $A$-paths as given in \cite{CrFe2}.

Then it is not hard to see that the operator $\partial$ defined in the theorem satisfies by construction:
$$\partial(\tilde{h}\cdot h)=\partial\tilde{h}\cdot\partial h.$$
This suggests one should think of $A$-homotopies as a ``groupoid'' over $P(A)$. In fact $A$-homotopies modded out by higher homotopies (so that the concatenation is associative) form a Banach groupoid over $P(A)$, and this groupoid presents a finite dimensional differentiable stack $\mathcal{G}(A)$. As shown in \cite{THZ}, $\mathcal{G}(A)$ is itself again a stacky groupoid with the same base as $A$ and it is in the sense of differentiable stacks the universal
integrating object of $A$.

\end{rem}
\begin{rem}The reader might be surprised that the formula \eqref{partial:path} does not
involve the values of $\omega(a,b)_{s,\epsilon}$ at the base points $\gamma^\epsilon(s)$ but at   $\phi_{s,1}^{a^\epsilon}(\gamma^{\epsilon}(1))$. Let us explain this fact: for each $\epsilon$, there is an $A_E$-homotopy between $a^\epsilon$ and the concatenation $h(a^\epsilon_B)\cdot a^\epsilon_K$ (this can be proved using Proposition \ref{hgeom}). Here $h(a^\epsilon_B)$ is the only horizontal $A_E$-path over $a^\epsilon_B$ that can be concatenated with $a^\epsilon_K$ that is, $h(a_B)$ has base path $s \to\phi^{a^\epsilon}_{s,0}(\gamma_K^\epsilon(1))$ which precisely coincides with $\phi_{s,1}^{a^\epsilon}(\gamma^{\epsilon}(1))$ by construction.

Thus we see that the homotopy condition we get in Theorem \ref{thm-2cocycle} relies on the representatives up to $A_E$-homotopy $h(a^\epsilon_B)\cdot a^\epsilon_K$ of $a^\epsilon$, rather than on $a^\epsilon$ itself.
\end{rem}

\subsection{The Monodromy Groupoid}\label{monodromy:grpd}  

We now turn to the study of the sequence of groupoids
\begin{equation}\label{exactgrpd}
1\to \G(\K)\xrightarrow{\tilde{\iota}} \G(A_E) \xrightarrow{\tilde{\pi}} \G(A_B)\to 1.
\end{equation}
obtained by integration (see \cite{MM}, \cite{MX}) of a Lie algebroid extension
\begin{equation}
  \K \stackrel{i}{\hookrightarrow} A_E \stackrel{\pi}{\twoheadrightarrow} A_B.
\end{equation}

Recall that $\tilde{\iota}$ and $\tilde{\pi}$ are defined at the level of paths, as follows:
$$\begin{array}{ccl}
\tilde{\iota}([a_K]_{\K})&:=&[i \circ a_K]_{A_E}, \\
\tilde{\pi}([a]_{A_E})&:=&[\pi\circ a]_{A_B},
\end{array}$$
for any $\tilde{a}_K\in P(\K)$, and $a\in P(A)$. Our aim in this section is to see why the sequence of groupoids (\ref{exactgrpd}) might not be exact in general. In other words, we want to explain the lack of exactness of the \emph{integration functor}. We start by showing that $\tilde{\pi}$ is surjective provided there is a complete Ehresmann connection.
\begin{prop}
 Let  $\K\hookrightarrow A_E \twoheadrightarrow A_B$ be a Lie algebroid extension that admits a complete Ehresmann connection.
 Then the sequence  of groupoids (\ref{exactgrpd}) is exact at $\G(A_B)$.
\end{prop}
\begin{proof}
 We have to prove that $\tilde{\pi}$ is surjective. Given any $A_B$-path $a_B\in P(A_B)$ covering $\gamma_B$ in $B$, one extends $a_B$ into a time-dependant section  $\alpha_B$ of $A_B$, and consider it horizontal lift $h(\alpha_B)$. Pick an arbitrary point $e_0$ in the fiber over $\gamma_B(0)$, then clearly $a(t):=h(\alpha_B)_{\phi^{a_B}_{t,0}(e_0)}$ defines a $A_B$-path that projects onto $a_B$. Completeness of the Ehresmann connection insures us that $a$ is well defined for all $t\in I$.
\end{proof}
\begin{rem}
 We stress the fact that completeness is needed in order to ensure surjectivity, otherwise one might only get a quasi-surjective
 morphism in the sense that for any $x\in\G(A_B)$, there exists $g_1,\dots, g_n\in \G(A_E)$ such that the composition  $\tilde{\pi}(g_1)\cdots\tilde{\pi}(g_n)=x$.
 In general though, the $g_i$'s might not be chosen to be composable, as the following examples from \cite{Wein3} show.
\end{rem}
\begin{ex}\emph{(Stairway To Heaven).}
Consider $E=\bigl\{(x,y)\in \mathbb{R}^2,\  y\in \mathbb{Z},\  x\in]y-1,y+1[\ \bigr\}$, it is a smooth submanifold of $\mathbb{R}^2$, and the projection on the first factor is a submersion onto $B:=\mathbb{R}$. Thus one gets a well defined algebroid morphism $TE\to TB$, that comes with a unique Ehresmann connection. It is easily seen that the induced groupoid morphism is only quasi-surjective.
\end{ex}

\begin{ex}\emph{(Highway To Hell).}
Consider now on $E=\mathbb{R}^2$ the integrable regular foliation $\F\subset TE$ spanned by
$\partial x-\exp(y)\partial y$. Denote $p$ the projection on the first factor. The restriction of
$dp$ to the Lie algebroid $A_E:=\F$ defines a surjective Lie algebroid morphism onto $A_B:=TB\to
B:=\mathbb{R}$. Again, we have a (unique) Ehresmann connection which is not complete: $h(\partial
x)=\partial x-exp(y)\partial y$. Lie algebroid morphism ${\d p_1}_{|\F}$ only integrates into a
quasi-surjective Lie groupoid morphism.
\end{ex}

\begin{prop}
 Let  $\K\hookrightarrow A_E \twoheadrightarrow A_B$ be a Lie algebroid extension that admits a complete Ehresmann connection. Then the sequence  of groupoids (\ref{exactgrpd}) is exact at $\G(A_E)$.
\end{prop}

\begin{proof}
By definition, elements in $\ker\ \tilde{\pi}$ are represented by $A_E$-paths whose projection on $A_B$ is homotopic to a trivial path. In particular, an element of the form $i\circ a_K$ projects onto a trivial $A_B$-path since  $\pi\circ i\circ a_K=0$. Thus we have $\im\  \tilde{\iota}\subset \ker\ \tilde{\pi}$ in a (very) trivial way.

Conversely, consider any element of $\ker\ \tilde{\pi}$ represented by some $A_E$-path $a$, and extend $a$ into a time-dependant section $\alpha$ as prescribed in last section. With the same notations, we get an identification $a=(a_K,a_B)$ where, by assumption $a_B$ is homotopic to a trivial path by means of some homotopy $h_B=a_B\d t+b_B \d\epsilon:TI^2\to A_B$.

One extends $b_B$ into a family of time-dependant sections $\beta_B$ of $A_B$ such that
${\beta_B}_{|t=0,1}=0$. Thus, if we let $\alpha^\epsilon$ be the solution of the evolution equation
$\frac{d\alpha}{d\epsilon}-\frac{dh(\beta_B)}{dt}=[\alpha,h(\beta_B)]$ with initial condition
$\alpha^{\epsilon=0}:=\alpha$, we get a $A_E$-homotopy $(a_K^\epsilon,a_B^\epsilon)$. By
construction $(a_K^1,a_B^1)=(a_K^1,0)$ which clearly represents an element in $\im \
\tilde{\iota}$. Thus we have shown that any element of $\ker\ \tilde{\pi}$ is
represented by a $A_E$-path of the form $(a_K,0_x)$, thus it is an element of $\im\ \tilde{\iota}$.
\end{proof}

We now see that, when there is a complete connection, the lack of exactness in the sequence \ref{exactgrpd} can only occur at $\G(\K)$. This means that $\G(\K)$ might not inject into $\G(A_E)$. In order to measure this, we introduce the following.
\begin{defn}
 We will call {\bf monodromy groupoid} the kernel of $\tilde{\iota}$, and denote it $\mathcal{M}$.
\end{defn}
We obtain by construction an exact sequence of groupoids:
\begin{equation}\label{exactmonodromy}
1\xrightarrow{} \M \xrightarrow{} \G(\K) \xrightarrow{} \ker\ \tilde{\pi} \xrightarrow{} 1.
\end{equation}

The next theorem explains how $\M$ is involved in the homotopy theories of $A_E$, $A_B$ and $\K$.
 In fact, it is the image of a connecting morphism $\partial_{2}:\pi_2(A_B)\ltimes_{B} E \to \G(\K)$,
 where $\pi_2(A_B)$ denote the homotopy classes of spheres in $A_B$:
\begin{thm} Let  $\K\hookrightarrow A_E \twoheadrightarrow A_B$ be a Lie algebroid extension that admits a complete Ehresmann connection. Then there exists a homomorphism
$$\partial_{2}:\pi_2(A_B)\ltimes_{B} E\to \G(\K),$$
that makes the following sequence exact:
$$ \cdots \to \pi_2(A_B)\ltimes E\xrightarrow{\partial_2} \G(\K)\xrightarrow{\tilde{\iota}} \G(A_E)\twoheadrightarrow \G(A_B)$$
\end{thm}
\begin{proof}
The boundary map $\partial_2$ can be constructed directly as follows: consider an element $x_0$ in $E$, $b_0$ its projection onto $B$, and $s_B$ a $A_B$-sphere based at $b_0$ i.e. an algebroid morphism $s_B=a_B\d t+b_B \d\epsilon:I\times I\to B$ satisfying ${a_B}_{|\epsilon=0,1}=0$, ${b_B}_{|t=0,1}=0$ (see the appendix \ref{app:spheres}).

Choose any  family $\beta^t_B$  of time-dependant sections of $A_B$ extending $b_B$, chosen such that $\beta_B^{t=0}=\beta_B^{t=1}=0$. Thus, the unique solution $\alpha_B$ of the evolution equation $[h(\beta_B),\alpha]_{A_E}=\frac{d}{d t}h(\beta_B)-\frac{d}{d\epsilon}\alpha$ with initial condition $\alpha^{\epsilon=0}=0$ induces at time $\epsilon=1$ a $\K$-path starting at $x_0$, which is homotopic (as a $A_E$-path) to the trivial path $0_{x_0}$ by construction. It is easily seen using the construction in the proof of Proposition \ref{prop:splithomotopy} that $a^1$ is a $\K$-path.

One defines $\partial_2(h_B, x_0)$ to be the $\K$-homotopy class of $a^1$. Clearly, $\partial_2(h_B,x_0)$ has source and target $x_0$, so it is an element in the isotropy $\G(\K)_{x_0}$.

We now show that $\partial(h_B,x_0)$ only depends on the homotopy class of $s_B$. Consider an
homotopy of $A_B$-spheres $s_B^u:TI^2\to A_B, u\in I$ between two spheres $s_B^{u=0}$ and
$s_B^{u=1}$; we just do the same construction as above with $u$ as a parameter: we extend $b_B^u$
into a smooth family $\beta_B^{u}$ of sections of $A_B$ vanishing whenever $t=0,1$ and we let
$\alpha^{\epsilon,u}(t)$ be the unique solution of
$$\begin{array}{ccc}
\frac{d}{dt}\alpha-\frac{d}{d\epsilon}h(\beta_B)=\left[h(\beta_B),\alpha\right]_{A_E},\\
\end{array}$$
 with initial condition $\alpha_{|\epsilon=0}=0$. We let $a_K^{u,\epsilon=1}$ the corresponding $\K$-path.

We now want to define a $\K$-homotopy between $a_K^{u=0,\epsilon=1}$ and $a_K^{u=1,\epsilon=1}$.
For this, we consider the unique solution of the evolution equation:
$$ \frac{d \theta}{d\epsilon}-\frac{dh(\beta_B)}{du}=[\theta,h(\beta_B)]_{A_E},$$
with initial condition $\theta_{|\epsilon=0}=0$. We know from the Lemma \ref{spherelemma} that ${\theta}$ also satisfies:
$$ \frac{d \theta}{dt}-\frac{d\alpha}{du}=[\theta,\alpha]_{A_E}.$$
and that $\theta_{|t=0,1}=0$ so it gives the desired homotopy.

To complete the proof, one only needs to apply theorem \ref{thm-2cocycle} to see that every element in $\M$ is obtained this way.
\end{proof}

\begin{rem}
 The operator defined above clearly coincides with the one defined in Theorem \ref{thm-2cocycle} applied to $A$-spheres. In fact, the arguments in the proof above show that $\partial(h)=\partial(\hat{h})$ provided $h$ and $\hat{h}$ are equivalent homotopies between two fixed $A_B$-paths $a^0_B$ and $a_B^1$.
 By equivalent homotopies, we mean by definition that there exists a Lie algebroid morphism  $h=a\d t+b \d\epsilon+c\d u:TI^3\to A$ with
 $a_{|u=0}\d t+b_{|u=0}\d\epsilon=h$, $a_{|u=1}\d t +b_{|u=1}\d\epsilon=\hat{h}$ and satisfying $c=0$ whenever $\epsilon$ or $t$ is in $\{0,1\}$.
\end{rem}

\subsection{Integrability}
The analogy between the monodromy groupoid described above and the monodromy groups
 of a Lie algebroid as constructed in \cite{CrFe2} is clear. In fact it is easily seen that they do coincide in the case of Atiyah exact sequences. Also, when $E$ is a tubular
neighborhood of an orbit $B$ of $A_E$ (see Example \ref{neighborhood}), the restriction of the
monodromy groupoid $\M$ to $B$ coincides with the monodromy (bundle of) groups along $B$ by
construction.

Just like monodromy groups control the integrability of a Lie algebroid (\cite{CrFe2}), in the case of clean extensions (see the Definition \ref{cleanextensions}
) the monodromy
groupoid controls the integrability of $A_E$.

\begin{defn}\label{discreteness} we will say that $\M$ is discrete near identities if, for any sequence $(g_n)_n\subset \M$ converging
to an identity $\textbf{1}_x$ in $\G(\K)$, we have $g_n=\textbf{1}_{x_n}$ for $n$ big enough.
\end{defn}

\begin{thm} Let $\K\hookrightarrow A_E \twoheadrightarrow A_B$ be a clean Lie algebroid extension.
Assume that it admits a complete Ehresmann connection and that both $\K$ and $A_B$ are
integrable Lie algebroids.

Then $A_E$ is integrable if $\mathcal{M}$ is discrete near identities in $\G(\K)$.
\end{thm}

\begin{proof}
For a generic Lie algebroid $A\to M$, we denote:
$$\N_x^A:=\{[g]\in \G(\g^A_x)\text{, g is }A\text{-homotopic to a trivial path}\},$$
where $g\in P(\g^A_x)$ and $x\in M$. According to \cite{CrFe2} we have an exact sequence:
$$ \N_x^A\hookrightarrow \G(\g^A_x) \twoheadrightarrow \G(A)_x^0. $$
Moreover, any element in $\N_x(A)$ can be represented by an element in $\Z(\g^A_x)$,
seen as a constant path (here, $\Z(\ker {\sharp_A}_x)$ denotes the center of the isotropy algebra $\g^A_x$ at $x$).
With these settings, integrability is equivalent to the following condition:
\renewcommand{\labelitemi}{}
\begin{itemize}
\item \emph{For any sequence $([v_n])$ in $\N^{A}_{x_n}$ represented by constant paths $v_n \in
Z(\ker{\sharp_{A}}_{x_n})$ that converges to a trivial path $0_{x}$ in $P(A)$, one necessarily has
 $v_n=0_{x_n}$ for $n$ big enough.}
\end{itemize}
So let $([v_n])\subset \N^{A_E}_{x_n}$ be a sequence where $v_n\in \Z(\g^E_{x_n})$ converges to a trivial path:
 $v_n\to 0_x$ for some $x\in E$. For each $n$, one can draw the following commutative diagram:
$$\xymatrix{\N^{\K}_{x_n} \ar@{^{(}->}[r]     & \G(\g^\K_{x_n})\ar[d]^{\tilde{\iota}}  \ar@{->>}[r]   &\G(\K)_{x_n}^0  \ar[d]  \\
            \N^{A_E}_{x_n}\ar@{^{(}->}[r]   & \G(\g^E_{x_n}) \ar[d]^{\tilde{\pi}}\ar@{->>}[r]   &\G(A_E)_{x_n}^0 \ar[d] \\
            \N^{A_B}_{y_n}\ar@{^{(}->}[r]          & \G(\g^B_{y_n})                     \ar@{->>}[r]
            &\G(A_B)_{y_n}^0,
}$$ where $y_n:=p(x_n)\in B$. As explained above all lines are exact. The middle column is not
exact but we still have $\im\ \tilde{\iota}= \ker \tilde{\pi}$.

Now consider the sequence of $A_B$-paths $\pi(v_n)$ in $\g^B_{y_n}$. The restriction of
$\pi$ to $\Z(\g^E_{x_n})$ is easily seen to induce an exact sequence
\begin{equation}\label{exactcenters} \Z(\g^{\K}_{x_n})\hookrightarrow \Z(\g^E_{x_n})
\twoheadrightarrow \Z(\g^B_{y_n}).\end{equation} Thus
$\tilde{\pi}([v_n])=[\pi(v_n)]$ is a sequence in $\N^{A_B}_{x_n}$ with $\pi(v_n)\in
\Z(\ker{\sharp_{A_B}}_{y_n})$ converging to $0_{y_n}$. Since $A_B$ is integrable, this implies
that $\pi(v_n)=0_{y_n}$ for $n$ big enough. So we see in \eqref{exactcenters} that $v_n\in
\Z(\g^\K_{x_n})$ and converges to $0_{x_n}$. By construction, we have $[v_n]_\K\in \M$
and converges to the identity $\textbf{1}_x$ in $\G(\K)$. The discreteness assumption on $\M$
  implies that $[v_n]_\K=\textbf{1}_{x_n}$ for $n$ big enough.
 Thus $[v_n]\in  \N^{\K}_{x_n}$ with $v_n\in \Z(\g^{\K}_{x_n})$ converging to $0_{x}$.
 We can conclude that $v_n=0_{x_n}$ for $n$ big enough because $\K$ is integrable.
\end{proof}

\appendix

\section{\texorpdfstring{The Weinstein groupoid and $A$-homotopies}{The Weinstein groupoid and A-homotopies}}
Let $A\to M$ be a Lie algebroid. We recall the construction of the Weinstein groupoid $\G(A)$ (see \cite{CrFe2} or \cite{CaFe} for an alternative approach in the Poisson case).
This is a topological groupoid, with source 1-connected fibers, which morally integrates $A$.
Indeed, $A$ is integrable if and only if $\G(A)$ is smooth and in this case $A(\G(A))$ is
canonically isomorphic to $A$.

We will denote by $P(A)$ the space of $A$-paths (up to reparametrization). By setting
$s(a)=p_A\circ a(0)$ and $t(a)=p_A\circ a(1)$, we shall think of $P(A)\tto M$ as an infinite
dimensional groupoid with multiplication given by concatenation (though units are not well
defined). On the groupoid $P(A)$ there is an equivalence relation $\sim$, called $A$-homotopy,
which preserves products, and one sets:
\[ \G(A):=P(A)/\sim . \]
The homotopy class of an $A$-path $a$ will be denoted $[a]_A$ or $[a]$ when no confusion seems possible.
Let us recall how $A$-homotopies are defined since this will be essential later. Suppose we
are given $\alpha^\epsilon$(t), a time dependent family of sections of $A$ depending on a parameter $\epsilon\in I:=[0,1]$, and $\beta^0(\epsilon)$ a time dependent section of $A$.
Then there exists a unique solution $\beta=\beta^t(\epsilon)$ of the following
\emph{evolution equation}:
\begin{equation}
\label{evolution}
\frac{d\alpha}{d\epsilon}-\frac{d\beta}{dt}=[\alpha,\beta],
\end{equation}
with initial condition $\beta^0(\epsilon)$. In fact, it is easily checked that the following integral formula provides a solution:
\begin{equation}
\label{hintegral}\beta^t(\epsilon):=\int_0^t(\psi^{\alpha^\epsilon}_{t,s})_*(\frac{d}{d\epsilon}\alpha^\epsilon(s))ds+(\psi_{t,0}^{\alpha^\epsilon})_*(\beta^0(\epsilon)).
\end{equation}
Here $\psi^{\alpha^\epsilon}$ denotes the flow of the time-dependant linear vector field on $A$, associated to the derivation $[\alpha^\epsilon,-]$ of sections of $A$ (see the appendix in \cite{CrFe2} for more details). We emphasize the use of the indices and parameters in the notation: we think of $\alpha$ as an $\epsilon$-family of $t$-time dependent sections of $A$, while we think of $\beta$ as a $t$-family of $\epsilon$-time dependent sections of $A$ (see why below).

The notion of $A$-homotopy is defined as follows. A family $a^\epsilon:I\mapsto A,\ \epsilon\in I$ of $A$-paths, over $\gamma^\epsilon:I\mapsto B$ is called a \emph{homotopy}
if $\gamma_\epsilon(0)$ in independent of $\epsilon$, and if the unique solution $\beta$ of equation (\ref{evolution}) with initial condition $\beta^0(\epsilon)=0$ satisfies:
\begin{equation}\label{hcondition}
\beta^1(\epsilon)_{\gamma^\epsilon(1)}=0, \forall \epsilon\in I.
\end{equation}
Here, $\alpha^\epsilon$ denotes any family of time-dependant sections of $A$ {\emph extending} $a$, that is, such that $\alpha^\epsilon_{\gamma^\epsilon(t)}(t)=a^\epsilon(t)$. One checks
that this definition is independent of the choice of $\alpha$ (see \cite{CrFe2}). We will refer to (\ref{hcondition}) as the {\emph homotopy condition}. Note that, if some $\alpha^0$ is fixed, and we are given $\beta$ with $\beta^0=\beta^1=0$, then equation (\ref{evolution}) can also be considered as an evolution equation for $\alpha$, which turns out to induce an homotopy.

The reason why we consider the evolution equation (\ref{evolution}) with non-vanishing initial condition is that this also leads to $A$-homotopies, as we now explain. Set $X=\sharp\alpha$ and $Y=\sharp\beta$, so that $X^\epsilon$ and $Y^t$ are families of time dependent vector fields on $M$ (which obviously satisfy an evolution equation in the algebroid $TM$). As the proof of next proposition shows, this forces their time-dependant flows $\phi^{X^\epsilon}_{t,0}, \phi^{Y^t}_{\epsilon,0}$ to be related as follows:
\begin{equation}\label{hflows}
\phi^{X^\epsilon}_{t,0}\circ\phi^{Y^0}_{\epsilon,0}=\phi^{Y^t}_{\epsilon,0}\circ\phi^{X^0}_{t,0}.
\end{equation}
In particular, if we denote by $\gamma^\epsilon(t)$ any of these two equivalent expressions applied to some $m_0\in M$, we obtain two families of $A$-paths $a^\epsilon$ and $b^t$, defined by $a^\epsilon(t):=\alpha^\epsilon(t)_{\gamma^\epsilon(t)}$ and $b^t:\epsilon\mapsto \beta^t(\epsilon)_{\gamma^\epsilon(t)}$. We have:

\begin{prop}
\label{hgeom}
For any couple $\alpha$ and $\beta$ satisfying the evolution equation (\ref{evolution}), the concatenations $a^1.b^0$ and $b^1.a^0$ defined above are homotopic $A$-paths.
\end{prop}

\begin{proof}
Assume that $A$ is integrable and denote by $\G(A)$ the source 1-connected Lie groupoid integrating $A$. Recalling that the Lie algebra of sections of $A$ can be identified
with the Lie algebra of right invariant vector fields on $\G$, we denote by $\ri{\alpha}^\epsilon$ and $\ri{\beta}^t_R$ the (time dependent) right invariant vector fields on $\G(A)$ that correspond
to $\alpha^\epsilon$ and $\beta^t$, respectively. Then the evolution equation (\ref{evolution})
exactly means that $\ri{\alpha}^\epsilon+\partial t$ and $\ri{\beta}^t+\partial \epsilon$ commute,
when seen as vector fields on $\G(A)\times I\times I$. This means that their flows:
$$
\begin{array}{cccc}
 \phi_u^{\alpha+\partial t}(g,t,\epsilon)&=&(\phi^{\ri{\alpha}^\epsilon}_{t+u,t}(g),t+u,\epsilon)\\
 \phi_v^{\beta+\partial\epsilon}(g,t,\epsilon)&=&(\phi^{\ri{\beta}^t}_{\epsilon+v,\epsilon}(g),t,\epsilon+v),
\end{array}
$$
commute, so we find:
$$
\phi^{\ri{\alpha}^{\epsilon+v}}_{t+u,t}\circ\phi^{\ri{\beta}^t}_{\epsilon+v,\epsilon}(g)=\phi^{\ri{\beta}^{t+u}}_{\epsilon+v,\epsilon}\circ \phi^{\ri{\alpha}^{\epsilon}}_{t+u,t}(g).
$$
In particular, taking $t=\epsilon=0$ and then switching the roles of $u,v$ with the ones of $t,\epsilon$, one gets:
$$
\phi^{\ri{\alpha}^{\epsilon}}_{t,0}\circ\phi^{\ri{\beta}^0}_{\epsilon,0}(1_{m_0})=
\phi^{\ri{\beta}^{t}}_{\epsilon,0}\circ \phi^{\ri{\alpha}^{0}_R}_{t,0}(1_{m_0}).
$$
Now, observe that the path in $\G$ corresponding to the concatenation $a^1.b^0$ is
the concatenation of the following paths
\[
\epsilon\mapsto  \phi^{\ri{\beta}^0}_{\epsilon,0}({1}_{m_0}),\quad t\mapsto\phi^{\ri{\alpha}^{1}}_{t,0}\circ\phi^{\ri{\beta}^0}_{1,0}({1}_{m_0}),
\]
while the one corresponding to $b^1.a^0$ is the concatenation of the paths
\[
t\mapsto\phi^{\ri{\alpha}^{0}}_{t,0}({1}_{m_0}),\quad \epsilon\mapsto\phi^{\ri{\beta}^{1}}_{\epsilon,0}\circ \phi^{\ri{\alpha}^{0}}_{1,0}({1}_{m_0}).
\]
These are clearly homotopic (in the $\s$-fibers) since they define the boundary of the square:
\[
(t,\epsilon)\mapsto \phi^{\ri{\alpha}^{\epsilon}}_{t,0}\circ\phi^{\ri{\beta}^0}_{\epsilon,0}({1}_{m_0})=\phi^{\ri{\beta}^{t}}_{\epsilon,0}\circ \phi^{\ri{\alpha}^{0}}_{t,0}({1}_{m_0}).
\]
We conclude that $a^1.b^0$ and $b^1.a^0$ are homotopic as $A$-paths.

In the case $A$ is not integrable, the above proposition still holds. The argument goes as follows: $a\d t+b\d\epsilon:TI^2\to A$ is a Lie algebroid morphism (as can be checked in local coordinates). Let $h:I^2\to I^2$ be any homotopy between $u^1\cdot v^0$ and $v^1 \cdot u^0$, where $u^1\cdot v^0$ is the concatenation of $v:\epsilon\mapsto (0,\epsilon)$ with $u^1:t\to(t,1)$ and $v^1 \cdot u^0$ the concatenation of $u^0:t\to (t,0)$ with $v^1:\epsilon\to(1,\epsilon)$. Then the composition  $(a\d t+b\d\epsilon)\circ dh:TI^2\to A$ is a Lie algebroid morphism that defines a homotopy between $a^1\cdot b^0$ and $b^1\cdot a^0$.
\end{proof}

\begin{rem}
We can also interpret the proposition in terms of $A$-homotopies (when $b^0=0$):
it says that $b^1$ is a representative (up to $A$-homotopy) of $a^1.(a^0)^{-1}$.
Hence, $b^1$ is trivial if and only if $a^0$ is homotopic to $a^1$.
\end{rem}

\section{Fibered Lie groupoids}\label{fiberedliegroupoids}
We briefly explain basic notions about groupoids acting on groupoids, and how to form semi-direct products.
\begin {defn} Given a submersion $p:E\to B$, a \textbf{fibered Lie groupoid} is a Lie groupoid $\G_V\tto E$ over $E$ such that $p\circ\s=p\circ \t$.
\end {defn}
Clearly, the orbits of a fibered Lie groupoid lie in the fibers of $p$, thus the restrictions ${\G_V}_{|E_b}\tto E_b$ are Lie groupoids as well for any $b\in B$.

We will denote $p:\G_V\to B$ rather than $p\circ\s$ or $p\circ \t$ (there is no real confusion to be done here). It is a Lie groupoid morphism onto the groupoid $B$ (whose only arrows are unities) so one can actually think of $\G_V$ as an extension of $B$ with trivial kernel.

\begin{defn} We will call the \textbf{gauge groupoid}, denoted $\Gau(\G_V)\tto B$, the set of all groupoid isomorphisms $\Phi_{b_2,b_1}:{\G_V}_{|E_{b_1}}\to {\G_V}_{|E_{b_2}}$,
 with source and target $\s(\Phi_{b_2,b_1})=b_1,\t(\Phi_{b_2,b_1})=b_2$ and obvious identities and composition.
\end{defn}
Of course the gauge groupoid is not Lie in general, however it is a nice intermediary in order to define actions of a Lie groupoid onto an fibered one.
\begin{defn} A smooth \textbf{action} of a Lie groupoid $\G_B\tto B$ on a fibered Lie groupoid $p:\G_V\to B$ is a groupoid morphism $\Phi:\G_B\to \Gau(G_V)$ covering the identity and such that
 \begin{eqnarray*}\G_B\times_B\G_V&\to& \G_V\\ (a_B,a_V)&\to& \Phi_{a_B}(a_V)
 \end{eqnarray*}
is a smooth map. Here, of course  $\G_B\times_B\G_V:=\{(a_V,a_B)\in \G_B\times\G_V,\ \s({a_B})=p(a_V)\}$ with the fibered topology.
\end{defn}
 Given smooth action of a Lie groupoid $\G_B\tto B$ on a fibered Lie groupoid $p:\G_V\to B$ there is a natural structure of Lie groupoid on  $\G_B\times_B\G_V\to E$. Source and target are given by $\s(a_B,a_V)=\s(a_V)$, $\t(a_B,a_V)=\t(\Phi_{a_B}(a_V))$ and composition by:
$$(a^2_B,a^2_K)\cdot(a^1_B,a^1_V)=(a^2_B\cdot a^1_B,\Phi_{a^1_B}^{-1}(a^2_V)\cdot a^1_V).$$
We will leave the details to the reader. Just note that this is a bit different from the usual notion of action of a groupoid on a submersion $\G_V\to E$ since the resulting groupoid is over $E$ rather than $\G_V$.

\section{Spheres in a Lie algebroid}\label{app:spheres}

We want to extend the notion of homotopy groups to algebroids. In the case of an integrable
algebroid $A\to M$ these should coincide with the second homotopy groups of the $\s$-fibers, seen as a
completely intransitive groupoid over $M$, however one needs a construction which is independent
of integrability.

\begin{defn} For a Lie algebroid $A$, we define a \textbf{$A$-sphere} to be an algebroid morphism $s=a\d t+b\d\epsilon:TI^2\to A$ such that $a_{|\epsilon=0,1}$ and $b_{|t=0,1}$ vanish (\emph{i.e} $a:I^2\to A$  (resp. $b:I^2\to A$) vanishes whenever $t\in\{0,1\}$ (resp. $\epsilon\in\{0,1\}$). The space of all such spheres will be denoted $S^2(A)$.
\end{defn}
Clearly, a $A$-sphere $s:TI^2\to A$ covers a topological sphere in the base manifold, that is, the base map of $s$ is a map $\gamma:I^2\to M$ whose restriction to the boundary of $I^2$ is reduced to a point $x_0$. We will say that $s$ is based at $x_0$.

If $A$ is integrable, then one can use the reasoning of proposition \ref{hgeom} to see that any $A$-sphere integrates to a unique sphere lying in the $\s$-fiber over its base point, and based at an identity. This motivates the following definition of homotopy for $A$-spheres.

\begin{defn} A \textbf{homotopy} of $A$-spheres is a smooth family of $A$-spheres $s_u:TI^2\to A, u\in I$ based at the same point $x_0\in M$.
\end{defn}

\begin{defn}
Two spheres $s^0,s^1:TI^2\to A$ are said to be homotopic if there exists a homotopy of spheres $s^u:TI^2\to A$ such that $s_i=s^i$, for $i=0,1$ 
$h=a\d t+b\d\epsilon+c\d u\to A$ such that $a^{u=i}\d t+b^{u=i}\d\epsilon=s^i$, for $i=0,1$.
\end{defn}
We get this way an equivalence relation on the space of $A$-spheres.

\begin{defn} We call second homotopy group, denoted $\pi_2(A)$ the space of $A$-spheres up to homotopy. 
\end{defn}

The above definition for homotopy might not look so natural. Let us work a little bit and see why this notion is correct.

\begin{prop}Given a family (parametrized by $u\in I$) of spheres $s^u=a^u\d t+b^u\d\epsilon:TI^2\to A$ based at a same point,
 there exists a unique $c:I^3\to A$ such that  $h:a\d t+b\d\epsilon+c\d u:TI^3\to A$ is an algebroid morphism and $c_{|t,\epsilon=0}=0$. Moreover, $c_{|t,\epsilon=1}=0$
\end{prop}
\begin{proof}
Let us first show uniqueness: assume $h:a\d t+b\d\epsilon+c\d u:TI^3\to A$ is an algebroid morphism, then one can extend $b$ into a time-dependant section $\beta$ of $A$ chosen such that $\beta_{|t=0,1}=0$, and let $\alpha$ be the solution of the evolution equation
$$\frac{d\alpha}{d\epsilon}-\frac{d\beta}{d t}=[\alpha,\beta]_A$$
with initial condition $\alpha_{|\epsilon=0}=0$. We already know that necessarily  $\alpha(\gamma^u(t,\epsilon))=a(\epsilon,t,u)$ (see \cite{CrFe2}).
For the same reason, if we let $\theta$ be the solution of the evolution equation
$$\frac{d\theta}{d\epsilon}-\frac{d\beta}{du}=[\theta,\beta]_A,$$
with initial condition $\theta_{\epsilon=0}=0$, then necessarily $\theta$ satisfies $\theta_{\gamma^u(t,\epsilon)}=c(t,\epsilon,u).$

To get existence, we have to make sure that if we let $c:=\theta_{\gamma^u(t,\epsilon)}$, then $h:adt+bd\epsilon+cdu$ is indeed a Lie algebroid morphism. By construction, one already has the following relations:
\begin{eqnarray*}
        \frac{d\alpha}{d\epsilon}-\frac{d\beta}{dt}&=&[\alpha,\beta]_A\\
        \frac{d\theta}{d\epsilon}-\frac{d\beta}{du}&=&[\theta,\beta]_A
 \end{eqnarray*}
 and, as can be checked in local coordinates, it is enough to show that:
$$\frac{d\theta}{dt}-\frac{d\alpha}{du}=[\theta,\alpha]_A$$
as well. We will need this result for other purposes, so we stated it in a separate lemma (see below).
There only remains to show that ${\theta_{\gamma^\epsilon(t)}}_{|\epsilon=1}=0$. For this, one can make the same reasoning as above switching the roles of $a$ and $b$, then invoke uniqueness.
\end{proof}
\begin{lem}\label{spherelemma}
 Let $\beta$ be a smooth family of sections of an algebroid $A$ with parameters $(\epsilon, t, u)\in I^3$ such that $\beta_{|t=0,1}$ vanishes. Denote $\alpha$ and $\theta$ the unique solutions of the evolution equations:
$$\frac{d\alpha}{d\epsilon}-\frac{d\beta}{d\epsilon}=[\alpha,\beta]_A,$$
$$\frac{d\theta}{d\epsilon}-\frac{d\beta}{du}=[\theta,\beta]_A,$$
with initial conditions $\alpha_{|\epsilon=0}=0$ and $\theta_{|\epsilon=0}=0 $. Then the following assertions hold:
\begin{enumerate}[(i)]
\item $\theta$ is solution of:
$$    \frac{d\theta}{dt}-\frac{d\alpha}{du}=[\theta,\alpha]_A,$$
\item $\theta$ vanishes whenever $t=0,1.$
\end{enumerate}
\end{lem}
\begin{proof}
 Define a family of sections $\phi$ as the difference:
$$\phi:=\frac{d\theta}{dt}-\frac{d\alpha}{du}-[\theta,\alpha]_A.$$
Clearly, $\phi_{|\epsilon=0}=0$ since $\alpha_{|\epsilon=0}$ and $\theta_{|\epsilon=0}$ both vanish. On the other hand, a straightforward computation shows that $\phi$ also satisfies:
$$\frac{d\phi}{d\epsilon}=[\beta,\phi]_A.$$
The unique solution of such an equation with initial condition $\phi_{|\epsilon=0}=0$ being $\phi=0$, we have proved $(i)$.

The fact that $\theta_{|t=0,1}=0$ is in fact straightforward: if one sees the equation $$\frac{d\theta}{d\epsilon}-\frac{d\beta}{du}=[\theta,\beta]_A,$$
as a family of equations with parameter $t$, clearly for $t=0,1$, since $\beta$ vanishes, $\theta$ has to vanish as well, which leaves $(ii)$ proved.\end{proof}

This way, we get an alternative definition for homotopies of $A$-spheres that should be more natural (though redundant).

\begin{defn}
Two $A$-spheres $s^0,s^1:TI^2\to A$ are homotopic if there exists a Lie algebroid morphism $s:a\d t+b\d\epsilon+c\d u:TI^3$ with $(a\d t+b\d\epsilon)_{|u=i}=s^i$ for  $i=0,1$, and satisfying 
\begin{enumerate}[i)]
 \item $a_{|\epsilon=0,1}=0$ and $b_{|t=0,1}=0$ (\emph{i.e.,} $a\d t+b\d\epsilon$ is an $A$-sphere for any $u\in I$).
\item $c_{|t,\epsilon=0,1}=0$ (\emph{i.e.}, c vanishes whenever $t\in\{0,1\}$ \emph{or} $\epsilon\in\{0,1\}$).
\end{enumerate}
\end{defn}

\bibliographystyle{amsalpha}

\begin{thebibliography}{A}

\bibitem{BrFe} O.~Brahic and R.~L.~Fernandes, Poisson Fibrations and 
	Fibered Symplectic Groupoids, Proceedings of the Conference on Poisson 
	Geometry in Mathematics and Physics, Tokyo, 2006,  AMS 
	Contemporary Mathematics Series, 2008, no. 450, pages 41-60.

\bibitem{BrFe2} O.~Brahic and R.~L.~Fernandes, \emph{Integration of coupling Dirac structures}, in preparation.



\bibitem{CW}  A. Cannas da Silva,   A. Weinstein,   \emph{Geometric models for noncommutative algebras},
Berkeley Math. Lecture Notes, 10, Amer. Math. Soc. (1999).

\bibitem{CaFe} A.S.~Cattaneo and G.~Felder, Poisson sigma models and
  symplectic groupoids, in \emph{Quantization of Singular Symplectic
  Quotients}, (ed. N.~P.~Landsman, M.~Pflaum, M.~Schlichenmeier),
  Progress in Mathematics \textbf{198} (2001), 41--73.

\bibitem{Cr} M.~Crainic, Prequantization and Lie brackets, \emph{J.~Symplectic Geometry}, Vol.2, No.4, 579-602, 2005.

\bibitem{CrFe2} M.~Crainic and R.~L.~Fernandes, Integrability of Lie
  brackets, \emph{Ann.~of Math.~(2)} \textbf{157} (2003), 575--620.



\bibitem{CrFe3} M.~Crainic and R.~L.~Fernandes, \emph{Exotic Characteristic Classes of Lie Algebroids}, Quantum Field Theory and Noncommutative Geometry, Lecture Notes in Physics, Vol. 662. Eds. Carow-Watamura, Ursula; Maeda, Yoshiaki; Watamura, Satoshi, Springer-Verlag, Berlin, 2005.

\bibitem{CZ} M.~Crainic and  C.~Zhu, \emph{Integrability of Jacobi and Poisson structures},
Annales de l'institut Fourier, 57 no. 4 (2007), p. 1181-1216.





\bibitem{Ehr} C.~Ehresmann, Les connexions infinit\'esimales dans un
  espace fibr\'e diff\'erentiable, \emph{S\'eminaire Bourbaki}, Vol.~1,
  Exp.~ No.~24, 153--168, Soc.~Math.~France, Paris, 1995.

\bibitem{Gi} V.~Ginzburg, Equivariant Poisson cohomology and a spectral sequence associated with a moment map, \emph{Int. J. Math.}, {\bf 10} (1999), 977-1010.

\bibitem{GL} V.~Ginzburg, J.~H.~Lu, Poisson cohomology of Morita equivalent Poisson manifolds,  \emph{IMRN}, 10  (1992), 199-205.



\bibitem{HM} P.J.~Higgins, K.~Mackenzie, \emph{Algebraic constructions in the category of Lie algebroids}, J. Algebra, 129 (1990), 194-230.



\bibitem{M} K. Mackenzie, \emph{Lie groupoids and Lie algebroids in differential geometry}, Cambridge Univ.
Press (1987).

\bibitem{MM} J.~Mrcun and I.~Moerdijk, \emph{On integrability of infinitesimal actions}, Amer. J. Math. 124 (2002), 567-593.
 
\bibitem{MX} Mackenzie and P. Xu P. \emph{Integration of Lie bialgebroids}, Topology, Volume 39, Number 3, May 2000 , pp. 445-467(23)

\bibitem{Pr}  J. Pradines,   \emph{Th{\'e}orie de Lie pour les groupo{\"\i}des diff{\'e}rentiables. Calcul diff{\'e}rentiel dans la cat{\'e}gorie des groupo{\"\i}des infinit{\'e}simaux}  C.R. Acad. Sci. Paris , \textbf{264} A  (1967),
245-248.

\bibitem{THZ} H.~Tseng and C.~Zhu, \emph{Integrating Lie algebroids via stacks}, Composition Mathematica 142 (2006), no.1, 251-270.

\bibitem{Vor} Y.~Vorobjev, \emph{Coupling tensors and Poisson geometry
  near a single symplectic leaf},Banach Center Publ.~\textbf{54}
  (2001), 249--274.

\bibitem{Wa} A.~Wade, \emph{Poisson fiber bundles and coupling Dirac structures},
    Ann Glob Anal Geom (2008) 33, 207-217.




\bibitem{Wein3} A.~Weinstein, \emph{Linearization of regular proper groupoids}, J. Inst. Math.
  Jussieu {1} (2002), no.3, 493-511.

\bibitem{X} P.~Xu, \emph{Poisson cohomology of regular Poisson manifolds}, Annales de l'institut Fourier, 42 no. 4 (1992), p. 967-988.

\end{thebibliography}

\end{document}